\documentclass[preprint]{elsarticle}

\biboptions{sort&compress}

\usepackage{mymacros}
\usepackage{driftdiffusionmacros}
\usepackage{mathtools}
\usepackage{subcaption}
\usepackage{enumitem}
\usepackage{float}
\usepackage{tikz}
\usepackage{booktabs}
\usepackage{pgfplots}
\usepackage{csvsimple}
\usepackage{tabularx}
\usepackage{array}
\newcolumntype{L}{>{\centering\arraybackslash}m{3cm}}
\usepackage{xspace}
\usepackage{comment}
\usepackage{soul, xcolor}

\usetikzlibrary{calc}
\pgfplotsset{compat=newest}
\pgfplotsset{width=6cm}
\usepgfplotslibrary{colorbrewer}
\usepackage{fullpage}

\newcounter{custom}
\numberwithin{custom}{section} 

\let\oldsubsection\subsection
\renewcommand{\subsection}{\stepcounter{custom}\oldsubsection}

\theoremstyle{plain}

\newtheorem{lemma}[custom]{Lemma}
\newtheorem{corollary}[custom]{Corollary}

\newtheorem{remark}[custom]{Remark}

\usetikzlibrary{intersections, pgfplots.fillbetween}

\definecolor{mygrey}{gray}{0.7}
\definecolor{negativecolour}{rgb}{0.85, 0.0, 0.0}
\definecolor{positivecolour}{rgb}{1.0, 0.7, 0.0}

\newcommand{\HH}{\ensuremath{\operatorname H}\xspace}

\newcommand{\sobh}[1]{\ensuremath{\HH^{#1}}}
\newcommand{\ensemble}[2]{\ensuremath{\left\{ #1:\;#2 \right\}}}
\renewcommand{\d}{\ensuremath{\,\mathrm{d}}}

 \providecommand{\D}{\ensuremath{\mathrm{D}}}
\newcommand{\abs}[1]{\ensuremath{\left|#1\right|}}
\usepackage{amsmath}
\newcommand{\Forall}{\:\forall\:}

\newcommand{\Foreach}{\quad\Forall}
\newcommand{\W}{\ensuremath{\Omega}\xspace}
\renewcommand{\jump}[1]{\ensuremath{\left\llbracket #1\right\rrbracket}}
\newcommand{\ujump}[1]{\ensuremath{\left\lfloor\!\!\left\lfloor #1 \right\rfloor\!\!\right\rfloor}}
\newcommand{\upen}{\mu}
\newcommand{\fes}{\ensuremath{\mathbb V}}
\newcommand{\qp}[1]{\ensuremath{\!\left({#1}\right)}}
\newcommand{\enorm}[1]{\ensuremath{\left|\!\left|\!\left| #1 \right|\!\right|\!\right|}}
\renewcommand{\vec}[1]{\ensuremath{\boldsymbol{#1}}}

\renewcommand{\formb}[2]{b\!\left( #1, #2\right)}

\usepackage{mathrsfs}
\provideboolean{usemathrsfs}
\setboolean{usemathrsfs}{true}
\newcommand{\mathscript}{\mathscr}

\renewcommand{\point}{\vec x}

\newcommand{\cA}{\ensuremath{\mathscript A}\xspace}

\begin{document}

\title{A nodally bound-preserving discontinuous Galerkin method for the drift-diffusion equation}

\author[1]{Gabriel R. Barrenechea}
\ead{gabriel.barrenechea@strath.ac.uk}

\author[2]{Tristan Pryer}
\ead{tmp38@bath.ac.uk}

\author[2]{Alex Trenam\corref{cor1}}
\ead{amt83@bath.ac.uk}

\cortext[cor1]{Corresponding author}

\affiliation[1]{organization={Department of Mathematics and Statistics},
  addressline={University of Strathclyde},
  city={26 Richmond Street, Glasgow},
  postcode={G1 1XH},
  country={UK},
}
\affiliation[2]{organization={Department of Mathematical Sciences},
  addressline={University of Bath},
  city={Claverton Down, Bath},
  postcode={BA2 7AY},
  country={UK},
}

\begin{abstract}
  In this work, we introduce and analyse discontinuous Galerkin (dG)
  methods for the drift-diffusion model. We explore two dG
  formulations: a classical interior penalty approach and a nodally
  bound-preserving method. Whilst the interior penalty method
  demonstrates well-posedness and convergence, it fails to guarantee
  non-negativity of the solution. To address this deficit, which is
  often important to ensure in applications, we employ a
  positivity-preserving method based on a convex subset formulation,
  ensuring the non-negativity of the solution at the Lagrange
  nodes. We validate our findings by summarising extensive numerical
  experiments, highlighting the novelty and effectiveness of our
  approach in handling the complexities of charge carrier transport.
\end{abstract}


\maketitle

\section{Introduction}
\label{sec:introduction}

Charge carrier transport is a physical process concerned with the
interaction of mobile charge carriers in the presence of an electric
field. Modelling this process is of interest in many applications
including battery electrolytes, fuel cells, ion membrane channels,
plasma physics, semi-conductor devices, etc. (see
e.g. \cite{Jungel2009, Wei2012} for details). The typical scale of
applications means that molecular dynamics approaches are often
computationally impractical, despite offering the most detailed
physical description. Continuum models are a common alternative, and
include the Poisson-Nernst-Planck (PNP) system, also known as the
drift-diffusion equations (or Van Roosbroeck equations in
semi-conductor literature).

In this work we introduce and analyse a new discontinuous Galerkin
(dG) method for the drift-diffusion model:
\begin{equation}
  \label{eq:drift-diffusion-strong}
    \partial_t\ddvar = \div{\nabla\ddvar + \ddvar\nabla\elpot},
\end{equation}
where $\elpot$ is a prescribed electric potential and $\ddvar$ the
particle concentration. We are interested in this as a prototypical example of an
electrolyte model that extends to the classical (normalised) two-species PNP equations
\begin{equation}
  \label{eq:PNP-strong}
  \begin{split}
    \partial_t \rho
    &=
    \div{\nabla\rho
      +
      \rho \nabla\elpot}
    \\
    \partial_t \nu
    &=
    \div{\nabla\nu
      -
      \nu \nabla\elpot}
    \\
    - \div{\perm\nabla\elpot}
    &=
    \rho - \nu + f,
  \end{split}
\end{equation}
where $\rho$ and $\nu$, respectively, represent positively and negatively
charged particle concentrations, and the electric potential $\elpot$ is now a variable \cite{Nernst1889,Planck1890,Debye1923}. The electric permitivity is $\perm > 0$, and $f \geq 0$ is the background fixed charge density. 

The drift-diffusion equation (\ref{eq:drift-diffusion-strong}) is an advection-diffusion equation, where the advection is the gradient of the electric potential. There is a rich body of literature concerning numerical methods (finite difference, finite volume, finite element, etc.) for advection-diffusion equations \cite{roos2008robust,Hundsdorfer2003}. In view of our application to concentrations (which must remain positive to make sense physically), we highlight the recent review of finite element methods respecting the discrete maximum principle for this problem \cite{barrenechea2024finite}. The presence of boundary and interior layers present a particular challenge due to the propensity of typical finite element solutions to display spurious oscillations around (pseudo-)discontinuities. In many cases these oscillations cause the solution to become negative.

Maintaining the positivity of concentration variables in the PNP system (\ref{eq:PNP-strong}) is not only desirable for physical interpretations but is integral to the stability of the solution and the satisfaction of an associated energy decay \cite{Gajewski1985}. Boundary layers, known physically as electrical double layers, can occur close to charged surfaces due to an exponential decay in electric potential \cite{Bagotsky2005}. The associated large values of $\nabla\elpot$ lead to locally convection-dominated flow, an effect which is only exacerbated with coupling of additional equations governing fluid flow, such as in the case of the Navier--Stokes--PNP system \cite{Constantin2019}.

A piecewise linear continuous Galerkin approximation of
(\ref{eq:PNP-strong}) is shown in \cite{Prohl2009} to satisfy an
M-matrix property, which ensures the satisfaction of a discrete
maximum principle. With dG methods there is no such guarantee, even
for an interior penalty discretisation of pure diffusion
\cite{barrenechea2024finite}. For the steady-state drift-diffusion
scheme discretised with conforming piecewise linear finite elements,
in \cite{BMP89} different monotone finite element methods were
proposed, all of them sharing the aim of rewriting the popular
Sharfetter-Gummel finite volume scheme in a finite element way. The
analysis of the methods is carried out by rewriting the problem as a
mixed formulation. Some other methods, such as the dG schemes of
\cite{Liu2017a} and \cite{Liu2022}, make use of positivity-preserving
limiters to deal with this problem. Other approaches involve making a
judicious change of variables to guarantee the positivity of the
concentrations. Examples include the logarithmic change of variables
used in \cite{Metti2016, Fu2021} and the the Slotboom transformation
utilised in \cite{Slotboom1973, Xie2013}. The cost of making such
substitutions is the introduction of further non-linearities to the
already-non-linear system.  In this paper we consider dG methods,
which are well-suited to advection-dominated regimes due to the
additional stability they offer.

With the PNP system in mind, we examine in this paper two methods for
solving (\ref{eq:drift-diffusion-strong}). The first is a classical dG
formulation of the problem, and the second is a method where
positivity of the concentration is hard-coded into the finite element
space. For the classical methodology, under both the natural advection-diffusion assumption $\Delta\psi \leq 0$ and, alternatively, a time step restriction, we are able to show
well-posedness and conduct a convergence analysis, however, the
methodology does not satisfy a positivity preservation property. That
is, there is no guarantee that the solution does not become negative
within the domain.

The second method aims to rectify this. The main idea
behind the positivity enforcement is to utilise the methodology from
\cite{Barrenechea2023,Amiri2024}, where we define a closed convex
subset of the standard dG space. Using this we formulate the problem
as a variational inequality at each timestep. This guarantees that the
dG solution is positive at the Lagrange nodes of each element. We show the problem is
well-posed and examine some of the qualitative properties of the
solution. In our numerical experiments we then show that the second
method preserves nodal postivity in the context of the PNP system.

The remainder of this paper is structured as follows: in
\S\ref{sec:pde} we fix some basic notation and discuss the model
problem its properties; a temporal semi-discretisation is studied in \S\ref{sec:semi-disc} to transform the parabolic problem into a sequence of elliptic problems; in \S\ref{sec:dg-disc} and \S\ref{sec:nodal-dg-disc} we then introduce a dG spatial discretisation and a nodally bound-preserving extension; \S\ref{sec:numerics} is devoted to demonstrating the properties possessed by these discretisations through numerical experiments; and finally, \S\ref{sec:conclusions} contains
some concluding remarks.

\section{The drift-diffusion equation}
\label{sec:pde}

Let $\Omega \subset \mathbb{R}^d$, $d\leq3$ be a bounded, polytopal
domain with Lipschitz boundary $\bound$. Throughout this work we denote the
standard Lebesgue spaces by $\L{p}{\domain}$, $1\le p\le \infty$,
$\omega\subset\mathbb{R}^d$, $d=1,2,3$, with corresponding norms
$\|\cdot\|_{\L{p}{\omega}}$. The $L^2$ inner product over $\omega$ is denoted $\ip{\cdot}{\cdot}_{\omega}$, where the subscript is omitted when $\omega = \domain$. We introduce the Sobolev spaces
\cite[c.f.]{Evans2010,Renardy2004}
\begin{equation}
  \Sob{m}{p}{\domain}
  := 
  \ensemble{w\in\L{p}{\domain}}
  {\D^{\vec\alpha}w\in\L{p}{\domain}, \text{ for } \abs{\vec\alpha}\leq m},
\end{equation}
which are equipped with norms and semi-norms
\begin{gather}
  \norm{w}{\Sob{m}{p}{\domain}}^2 
  := 
  \sum_{\abs{\vec \alpha}\leq m}\norm{\D^{\vec \alpha} w}{\L{p}{\domain}}^2 
  \text{ and }
  \abs{w}_{\Sob{m}{p}{\domain}}^2 
  =
  \sum_{\abs{\vec \alpha} = m}\norm{\D^{\vec \alpha} w}{\L{p}{\domain}}^2,
\end{gather}
respectively, where $\vec\alpha = \{ \alpha_1,...,\alpha_d\}$ is a
multi-index, $\abs{\vec\alpha} = \sum_{i=1}^d\alpha_i$ and
derivatives $\D^{\vec\alpha}$ are understood in a weak sense.
We identify the Hilbertian Sobolev spaces and norms by the notation $\H{m}{\domain} := \Sob{m}{2}{\domain}$, and a zero subscript (e.g. $\Honezero{\domain}$) indicates vanishing trace on $\bound$.

We consider the drift-diffusion equation, an advection-diffusion
problem where the advection is the gradient of a potential $\elpot =
\elpot\!\left(\point, t\right)$. In the context of electrolytes, this
potential can be understood physically as the global electric
potential, and $\nabla\elpot$ is the electric field. In this viewpoint
the solution variable $\ddvar = \ddvar\!\left(\point, t\right)$
represents the density of some charged particles, with an initial
concentration profile given by $\ddvarinit$. For simplicity of
presentation, we suppose homogeneous Dirichlet boundary condition on
$\bound\times\timedom$, however, our results extend in a
straightforward fashion to more general Dirichlet conditions (see
\cite{Amiri2024,Barrenechea2023}). We explore such numerical
examples in Section \ref{sec:numerics}. Therefore, we seek $\ddvar$
satisfying
\begin{equation}
  \label{eq:IBVP}
  \begin{split}
    \partial_t\ddvar = \div{\nabla\ddvar + \ddvar\nabla\elpot} &\text{ in }\domain\times\timedom, \\
    \ddvar = 0 &\text{ on }\bound\times\timedom, \\
    \ddvar = \ddvarinit &\text{ on }\domain\times\left\{0\right\}.
  \end{split}
\end{equation}
Introducing the bilinear forms $a : \Honezero{\domain}\times\Honezero{\domain}\longrightarrow\reals$ and $b : \Ltwo{\domain}\times\Honezero{\domain}\longrightarrow\reals$, defined by
\begin{align}
  \forma{w}{\ddtest} &:= \int_\W \nabla w \cdot \nabla\ddtest \d \vec x, \\
  \formb{w}{\ddtest} &:= \int_\W \qp{w\nabla\elpot} \cdot {\nabla\ddtest} \d \vec x,
\end{align}
the initial-boundary value problem (\ref{eq:IBVP}) can be written weakly as follows. Given $\elpot\in\Sob{2}{\infty}{\domain}$ and initial data $0\leq \ddvarinit\in\L\infty{\domain}$, seek $\ddvar\!\left(t\right)\in\Honezero{\domain}$, for almost every $t\in\timedom$, such that
\begin{equation}
  \label{eq:drift-diffusion}
  \ip{\partial_t\ddvar}{\ddtest}
  +
  \forma{\ddvar}{\ddtest}
  +
  \formb{\ddvar}{\ddtest} = 0\quad\quad\forall\ddtest\in\Honezero{\domain}.
\end{equation}
\begin{remark}
  The regularity assumption $\elpot\in\Sob{2}{\infty}{\domain}$ is strong; however, much of the analysis can still be conducted under weaker assumptions. In particular, $\elpot \in \Sob{2}{p}{\domain}$ for $p > d$ ensures that $\nabla \elpot \in \Linf{\domain}$, which is sufficient to prove most of the results presented herein. To avoid technical diversions, we focus on the more regular case in this work.
\end{remark}

Standard existence and uniqueness results for advection-diffusion equations (see \cite[c.f.]{Evans2010,Renardy2004}) apply to establish the well-posedness of \eqref{eq:drift-diffusion}.
\begin{lemma}[PDE well-posedness]
  \label{lem:drift-diffusion-well-posedness}
  Let $\elpot\in\Sob{2}{\infty}{\domain}$ and $0 \leq\ddvarinit \in
  \L{\infty}{\domain}$. If $\Delta\elpot \leq 0$, then there exists a unique solution $\ddvar(t) \in
  \Bochner{L^2}{\timedom}{\Honezero{\domain}} \cap
  \Bochner{H^1}{\timedom}{\H{-1}{\domain}}$ to the weak IBVP
  \eqref{eq:drift-diffusion}.
\end{lemma}

The problem (\ref{eq:drift-diffusion}) satisfies a stability result and a parabolic maximum principle, which we now detail. The preservation of these properties at the discrete level is the focus of the following sections. We begin with a useful lemma.
\begin{lemma}[Energy identity]
  \label{lem:Laplacian-sign}
  Let $w\in\Honezero{\domain}$ and
  $\elpot\in\Sob{2}{\infty}{\domain}$. Then
  \begin{equation}
    \forma{w}{w}
    +
    \formb{w}{w}
    =
    \norm{\nabla w}{\Ltwo{\domain}}^2
    -
    \half\ip{\Delta\elpot}{w^2}.
  \end{equation}
\end{lemma}
\begin{proof}
  Using the definition of $\formb{\cdot}{\cdot}$ and integrating by parts yields
  \begin{align}
    \formb{w}{w}
    &=
    \ip{w\nabla\elpot}{\nabla w} \\
    &=
    -\ip{\div{w\nabla\elpot}}{w}.
  \end{align}
  Using the product rule and rearranging we obtain
  \begin{equation}
    \formb{w}{w} = -\half\ip{\Delta\elpot}{w^2}.
  \end{equation}
  The result then follows from the definition of $\forma{\cdot}{\cdot}$.
\end{proof}

\begin{lemma}[Stability]
  \label{lem:drift-diffusion-pde-energy}
  For almost every $t\in\timedom$, let
  $\ddvar\!\left(t\right)\in\Honezero{\domain}$ solve
  (\ref{eq:drift-diffusion}), with $\elpot\in\Sob{2}{\infty}{\domain}$ and
  $0 \leq \ddvarinit\in\L{\infty}{\domain}$. Then we have
  \begin{align}
    \frac{d}{dt}\left[\half\norm{\ddvar}{\Ltwo{\domain}}^2\right] = -\norm{\nabla \ddvar}{\Ltwo{\domain}}^2 + \half\ip{\Delta\elpot}{\ddvar^2}.
  \end{align}
  The right hand side of the above equality is non-positive if $\Delta\elpot \leq 0$.
\end{lemma}
\begin{proof}
  Choosing $\ddtest = \ddvar$ in equation (\ref{eq:drift-diffusion}), we have
  \begin{equation}
    \frac{d}{dt}\left[\half\norm{\ddvar}{}^2\right] = \ip{\partial_t\ddvar}{\ddvar} = -\forma{\ddvar}{\ddvar} - \formb{\ddvar}{\ddvar},
  \end{equation}
  and applying Lemma \ref{lem:Laplacian-sign} completes the proof.
\end{proof}

The following result is a direct application of the parabolic maximum principle applied to the drift-diffusion equation \eqref{eq:IBVP}, under the assumption $\Delta\elpot \leq 0$.

\begin{lemma}[Parabolic Maximum Principle {\cite[Theorem 4.26]{Renardy2004}}]
  \label{lem:parabolic-maximum-principle}
  Let $\ddvar\in\Bochner{C}{\timedom}{\Honezero{\domain}}$ be a solution of the initial-boundary value problem (\ref{eq:IBVP}) with $\elpot\in\Sob{2}{\infty}{\domain}$ and $0 \leq \ddvarinit\in\Linf{\domain}$. Suppose $\Delta\elpot \leq 0$ in $\domain$. Then, for almost every $t\in\timedom$, the solution satisfies the maximum principle
  \begin{equation}
    0 \leq \inf_{\domain}\ddvarinit
    \leq
    \inf_{\domain}\ddvar\!\left(t\right)
    \leq
    \sup_{\domain}\ddvar\!\left(t\right)
    \leq \sup_{\domain}\ddvarinit.
  \end{equation}
\end{lemma}

\section{A temporal semi-discretisation}
\label{sec:semi-disc}

In this section we examine an implicit backward Euler discretisation
of the problem (\ref{eq:drift-diffusion}) and some of the properties
inherited by this discretisation. We semi-discretise in time, and
hence the parabolic problem is transformed into a sequence of elliptic
problems.

For $N\in\mathbb{N}$, we define the set of discrete time steps
$\left\{t^0, \ldots, t^N\right\}$, with $0 = t^0 < t^1 < \ldots < t^N
= T$ and denote the time step size by $\timestep$. A variable $w$ at
time step $t^n$ is denoted $w^n := w\left(t^n\right)$. Our arguments
work with variable $\timestep$, but for simplicity of presentation we
keep it fixed.

With the above notations (\ref{eq:IBVP}) can be presented as a sequence of elliptic
problems. Let $u^0 := u_0\in\L{\infty}{\domain}$, with $u_0 \geq 0$, and
$\elpot\in\Sob{2}{\infty}{\domain}$. For $n = 1, 2, \dots, N$ find
$u^n\in\Honezero{\domain}$ such that
\begin{equation}
  \label{eq:dd-semi}
  \cA(u^n, v)
  :=
  \ip{u^n}{v}
  +
  \tau
  \qp{
    a(u^n, v) + b(u^n, v)
  }
  =
  \ip{u^{n-1}}{v} \Foreach v\in \Honezero{\domain}.
\end{equation}

We now show that if $\Delta\elpot \leq 0$, i.e., under the same
assumption as Lemma \ref{lem:drift-diffusion-well-posedness}, then
(\ref{eq:dd-semi}) is well-posed. In addition, there is a time step
condition which, if satisfied, guarantees well-posedness for any
$\elpot\in\Sob{2}{\infty}{\domain}$. The proof is based on the following
coercivity and boundedness results regarding $\cA(\cdot, \cdot)$.

\begin{lemma}[Semi-discrete coercivity]
  \label{lem:time-disc-coerc}
  Let $w\in\Honezero{\domain}$ and $\elpot\in\Sob{2}{\infty}{\domain}$. If $\Delta\elpot \leq 0$, then
  \begin{equation}
    \label{eq:ass-pot}
    \cA(w, w)
    \geq
    \norm{w}{\Ltwo{\domain}}^2
    +
    \tau \norm{\nabla w}{\Ltwo{\domain}}^2.
  \end{equation}
  Alternatively, removing the assumption of $\Delta \psi$, suppose
  instead that
  \begin{equation}
    \label{eq:timestep-assumption}
    \timestep < \frac 4 {\constSob^2 \norm{\Delta \psi}{\L{3}{\domain}}^2}.
  \end{equation}
  where $\constSob$ is the constant associated with the Sobolev embedding $\sobh{1}_0(\domain)\hookrightarrow\L{6}{\domain}$, which holds for $d \leq 3$. Then
  \begin{equation}
    \label{eq:ass-timestep}
    \cA(w, w)
    \geq
    \frac 12
    \norm{w}{\Ltwo{\domain}}^2
    +
    \frac \tau 2 \norm{\nabla w}{\Ltwo{\domain}}^2.
  \end{equation}
\end{lemma}
\begin{proof}
  From the definition of $\cA(\cdot, \cdot)$ and Lemma \ref{lem:Laplacian-sign}, we have
  \begin{align}
    \cA(w, w) &= \norm{w}{\Ltwo{\domain}}^2 + \timestep\forma{w}{w} + \timestep\formb{w}{w} \\
              &= \norm{w}{\Ltwo{\domain}}^2 + \timestep\norm{\nabla w}{\Ltwo{\domain}}^2 - \frac{\timestep}{2}\ip{\Delta\elpot}{w^2}.
  \end{align}
  If $\Delta\elpot \leq 0$, then (\ref{eq:ass-pot})
  follows. Otherwise, we invoke the Sobolev embedding
  $\sobh{1}_0(\domain)\hookrightarrow\L{6}{\domain}$, for $d \leq 3$,
  and use H\"older's inequality to see that
  \begin{align}
    \frac{\timestep}{2}\ip{\Delta\elpot}{w^2} &\leq \frac{\timestep}{2}\norm{\Delta\elpot}{\L{3}{\domain}}\norm{w}{\L{6}{\domain}}\norm{w}{\Ltwo{\domain}} \label{eq:Laplacian-estimate-start} \\
    &\leq \constSob\frac{\timestep}{2}\norm{\Delta\elpot}{\L{3}{\domain}}\norm{\nabla w}{\Ltwo{\domain}}\norm{w}{\Ltwo{\domain}}.
  \end{align}
  Young's inequality then gives
  \begin{equation}
    \label{eq:Laplacian-estimate-end}
    \frac{\timestep}{2}\ip{\Delta\elpot}{w^2} \leq \constSob^2\frac{\timestep^2}{8}\norm{\Delta\elpot}{\L{3}{\domain}}^2\norm{\nabla w}{\Ltwo{\domain}}^2 + \frac 12 \norm{w}{\Ltwo{\domain}}^2,
  \end{equation}
  and therefore
  \begin{equation}
    \cA(w, w) \geq \frac 12\norm{w}{\Ltwo{\domain}}^2 + \timestep\left(1 - \constSob^2\frac{\timestep}{8}\norm{\Delta\elpot}{\L{3}{\domain}}^2\right)\norm{\nabla w}{\Ltwo{\domain}}^2.
  \end{equation}
  The result in (\ref{eq:ass-timestep}) is then a consequence of the time step restriction (\ref{eq:timestep-assumption}).
\end{proof}

\begin{lemma}[Boundedness of $\cA(\cdot, \cdot)$]
  \label{lem:dd-semi-boundedness}
  Let $w, v\in\Honezero{\domain}$ and $\elpot\in\Sob{2}{\infty}{\domain}$. Then
  \begin{equation}
    \cA(w, v) \leq \left(1 + C_P\norm{\nabla\elpot}{\L{\infty}{\domain}}\right)\left(\norm{w}{\Ltwo{\domain}}^2 + \timestep\norm{\nabla w}{\Ltwo{\domain}}^2\right)^{1/2}\left(\norm{v}{\Ltwo{\domain}}^2 + \timestep\norm{\nabla v}{\Ltwo{\domain}}^2\right)^{1/2},
  \end{equation}
  where $C_P$ is the Poincar\'e constant.
\end{lemma}
\begin{proof}
  Applying the H\"older and Poincar\'e inequalities to the definition of $\cA(\cdot, \cdot)$ gives
  \begin{align}
    \cA(w, v) &= \ip{w}{v} + \timestep\forma{w}{v} + \timestep\formb{w}{v} \\
              &\leq \norm{w}{\Ltwo{\domain}}\norm{v}{\Ltwo{\domain}} + \timestep\left(1 + C_P\norm{\nabla\elpot}{\L{\infty}{\domain}}\right)\norm{\nabla w}{\Ltwo{\domain}}\norm{\nabla v}{\Ltwo{\domain}} \\
              &\leq \left(1 + C_P\norm{\nabla\elpot}{\L{\infty}{\domain}}\right)\left(\norm{w}{\Ltwo{\domain}}\norm{v}{\Ltwo{\domain}} + \timestep\norm{\nabla w}{\Ltwo{\domain}}\norm{\nabla v}{\Ltwo{\domain}}\right). \label{eq:time-disc-boundedness-eq}
  \end{align}
  and using the discrete Cauchy-Schwarz inequality then yields the result.
\end{proof}

Combining Lemma \ref{lem:time-disc-coerc} and Lemma
\ref{lem:dd-semi-boundedness}, the Lax-Milgram Lemma then yields the
following result:
\begin{corollary}[Existence and uniqueness of semi-discrete solution]
  Under the conditions of Lemma
  \ref{lem:time-disc-coerc}, for $n = 1, 2,\ldots, N$, there exists a unique $u^n\in\Honezero{\domain}$
  solving (\ref{eq:dd-semi}).
\end{corollary}

The temporal semi-discretisation (\ref{eq:dd-semi}) satisfies the
following stability and maximum principle results, which are
semi-discrete analogues of the ones presented in Lemmata
\ref{lem:drift-diffusion-pde-energy} and
\ref{lem:parabolic-maximum-principle}, respectively.
\begin{lemma}[Semi-discrete stability]
 Let the conditions of Lemma \ref{lem:time-disc-coerc} be satisfied, with $\Delta\elpot \leq 0$, and let $u^n\in\Honezero{\domain}$ solve (\ref{eq:dd-semi}), for $n = 1, 2,\ldots, N$. Then
  \begin{equation}
    \half\norm{u^n}{\Ltwo{\domain}}^2
    \leq
    \half\norm{u^{n-1}}{\Ltwo{\domain}}^2    
    -
    \timestep\norm{\nabla u^n}{\Ltwo{\domain}}^2.
  \end{equation}
\end{lemma}
\begin{proof}
  From (\ref{eq:dd-semi}) and Lemma \ref{lem:time-disc-coerc} we have
  \begin{equation}
    \norm{u^n}{\Ltwo{\domain}}^2 + \timestep\norm{\nabla u^n}{\Ltwo{\domain}}^2 \leq \cA(u^n, u^n)  = \ip{u^{n-1}}{u^n},
  \end{equation}
  and the result follows by application of the Cauchy-Schwarz and Young inequalities.
\end{proof}

\begin{lemma}[Maximum Principle for the Semi-Discrete Problem]
  \label{lem:semi-discrete-maximum-principle}
  Suppose $\elpot\in\Sob{2}{\infty}{\domain}$ and $\Delta\elpot \leq 0$ in $\domain$. Let $\left\{u^n\right\}_{n=0}^N$ be the sequence of solutions to the semi-discrete problem (\ref{eq:dd-semi}) with initial condition $u^0 = u_0 \in \Linf{\domain}$, where $u_0 \geq 0$. Then the following maximum principle holds:
  \begin{equation}
    0 \leq \inf_{\domain} u^{0} \leq \inf_{\domain} u^n \leq \sup_{\domain} u^n \leq \sup_{\domain} u^{0}, \quad \text{for all } n = 1, 2, \dots, N.
  \end{equation}
\end{lemma}

\begin{proof}
  We prove the non-negativity of $u^n$ by induction on $n$. For the
  base case, $n = 0$, the initial condition $u^0 = u_0 \in
  \Linf{\domain}$ is given and, by assumption, $u_0 \geq 0$. Now,
  assume inductively that $u^{n-1} \geq 0$ for some $n \geq 1$. We
  want to show that $u^n \geq 0$.
  
  Let us begin by defining the positive and negative parts of $u^n$ by $\qp{u^n}_+ := \max\qp{u^n, 0}$ and $\qp{u^n}_- := u^n - \qp{u^n}_+$, respectively, noting that $\qp{u^n}_+, \qp{u^n}_-\in\Honezero{\domain}$ since $u^n\in\Honezero{\domain}$.
  Then the mutually-disjoint support of $\qp{u^n}_+$ and $\qp{u^n}_-$ implies that
  \begin{equation}
    \ip{u^n}{\left(u^n\right)_-}
    =
    \ip{\left(u^n\right)_+ + \left(u^n\right)_-}{\left(u^n\right)_-}
    =
    \norm{\left(u^n\right)_-}{\Ltwo{\domain}}^2.
  \end{equation}
  Choosing $v = \left(u^n\right)_-$ in
  (\ref{eq:dd-semi}) then leads to
  \begin{equation}
    \norm{\left(u^n\right)_-}{\Ltwo{\domain}}^2 = \ip{u^{n-1}}{\left(u^n\right)_-}
    -
    \timestep \left( \forma{u^n}{\left(u^n\right)_-}
    +
    \formb{u^n}{\left(u^n\right)_-} \right).
  \end{equation}
  It follows from the inductive hypothesis that
  \begin{equation}
    \ip{u^{n-1}}{\left(u^n\right)_-} \leq 0,
  \end{equation}
  and using again the mutually-disjoint
  support of $\qp{u^n}_+$ and $\qp{u^n}_-$ then gives
  \begin{equation}
    \norm{\left(u^n\right)_-}{\Ltwo{\domain}}^2
    \leq
    -\timestep\qp{
      \forma{\left(u^n\right)_-}{\left(u^n\right)_-}
      +
      \formb{\left(u^n\right)_-}{\left(u^n\right)_-}
    }.
  \end{equation}
  If $\Delta\elpot \leq 0$, then Lemma \ref{lem:Laplacian-sign} implies that
  \begin{equation}
    \norm{\left(u^n\right)_-}{\Ltwo{\domain}}^2 = 0,
  \end{equation}
  and so $\left(u^n\right)_- \equiv 0$. In other words $u^n \geq 0$.
  The upper bound can be shown in a similar fashion taking $v =
  \left(u^n - \norm{u^{0}}{L^\infty(\domain)}\right)_+$.
\end{proof}

\section{A discontinuous Galerkin method}
\label{sec:dg-disc}

Let $\mathcal T$ be a regular subdivision of $\W$ into disjoint
simplicial or box-type (quadrilateral/hexahedral) elements $K$. We
assume that the subdivision $\mathcal T$ is shape-regular, that
$\bar{\W}=\bigcup_{K\in\mathcal T}(\bar{K})$, and that the codimension-one elemental
faces are straight planar segments; these will be, henceforth,
referred to as \emph{facets}. We shall denote the
union of all \emph{interior} facets associated with
$\mathcal{T}$, not including the boundary, by $\mathcal{E}_{\rm{int}}$.

For $0 \leq p \in \mathbb N_0$ and a $K\in\mathcal T$ we denote the
set of polynomials of total degree at most $p$ by $\mathbb P^{p}(K)$ which
allows us to define the discontinuous Galerkin finite element space
\begin{equation}
  \fes_p := \ensemble{w_h \in \Ltwo{\domain}}{w_h\vert_{K} \in \mathbb P^{p}(K) \Foreach K\in\mathcal T, \text{ and with vanishing trace on }\bound}.
\end{equation}
\begin{remark}
  Our motivation for imposing the boundary conditions strongly is the pointwise nature of the contraints we shall apply in Section \ref{sec:nodal-dg-disc}. It is then straightforward to extend the method to non-homogeneous conditions. We believe the nodally bound-preserving method is compatible with discretisations with weakly-enforced boundary conditions, however, this remains the focus of active research.
\end{remark}

Let $K_1, K_2 \in \mathcal T$, with $K_1$ upwind from $K_2$, be two
elements sharing a facet $e\in \mathcal{E}_{\rm{int}}$. For a function
$w:\W\to\reals$ we define the jump and average operators over $e$ by
\begin{align}
  &\jump{w}_e
    :=
    w\vert_{K_1} \normal_{K_1}
    +
    w\vert_{K_2} \normal_{K_2},
  &&\ujump{w}_e
     :=
     w\vert_{K_1}
     -
     w\vert_{K_2},
  &&&\avg{w}_e := \half\left(w\rvert_{K_1} + w\rvert_{K_2}\right),
\end{align}
respectively. We will usually omit the subscript $e$. Let $h_K := \text{diam}(K)$ and $h \in \fes_0$ to be the piecewise constant meshsize function such that $h|_K = h_K$.

Now we define the discontinuous Galerkin method. Let $0 \leq u_0\in\L{\infty}{\domain}$, $\elpot\in\Sob{2}{\infty}{\domain}$, and $\Pi_h : \Ltwo{\domain}\longrightarrow\finspace_p$ be the $L^2$ projection. For $u_h^0= \Pi_h\!\left(u_0\right)$ and each $n=1, 2, \ldots, N$, find $u_h^n \in \fes_p$ such that
\begin{equation}
  \label{eq:dg-disc}
  \cA_h(u_h^n,v_h)
  :=
  \ip{u_h^n}{v_h}
  +
  \tau \qp{
    a_h(u_h^n ,v_h)
    +
    b_h(u_h^n, v_h)
  }
  =
  \ip{u^{n-1}_h}{v_h} \Foreach v_h \in \fes_p,
\end{equation}
where, for $w, v \in \H{2}{\mathcal T}$ and $\pen, \upen > 0$,
\begin{equation}
  a_h(w, v)
  :=
  \sum_{K\in\mathcal T}
  \qp{
    \int_K \nabla w \cdot \nabla v  \d \vec x
  }
  -
  \int_{\mathcal{E}_{\rm{int}}}
  \qp{
    \jump{w} \cdot \avg{\nabla v}
    +
    \jump{v} \cdot \avg{\nabla w}
    -
    \frac \sigma h \jump{w} \cdot \jump{v}
  } \d s
\end{equation}
and
\begin{equation}
  \label{eq:upwind-def}
  b_h(w,v)
  :=
  \sum_{K\in\mathcal T}
  \qp{\int_K w \nabla \psi \cdot \nabla v  \d \vec x}
  -
  \int_{\mathcal{E}_{\rm{int}}}
  \qp{
    \left(\nabla \psi \cdot \normal\right) \avg{w}\ujump{v}
    -
    \frac{\upen}{2}\lvert\nabla\elpot\cdot\normal\rvert\ujump{w}\ujump{v}
  }\d s.
\end{equation}

The bilinear form $a_h(\cdot, \cdot)$ is the symmetric interior penalty (SIP) discretisation of the diffusive term, which is coercive, assuming $\pen$ is chosen large enough, on $\finspace_p$ under the norm
\begin{equation}
  \norm{w_h}{\operatorname{sip}}^2 := \norm{\finnabla w_h}{\Ltwo{\domain}}^2 + \frac{\pen}{h}\norm{\jump{w_h}}{\Ltwo{\mathcal{E}_{\rm{int}}}}^2,
\end{equation}
where $\pen$ depends on the polynomial degree $p$ and geometric
features of the mesh \cite{cangiani2014hp}.
\begin{lemma}[{SIP coercivity, e.g. \cite[Lemma 4.12]{DiPietro2012}, \cite[Lemma 27]{Cangiani2017}}]
  \label{lem:sip-coerc}
  Let $w_h\in\finspace_p$. There exists $\widehat{\pen} > 0$ and $\constsipcoerc > 0$, such that if $\pen \geq \widehat{\pen}$, then
  \begin{equation}
    a_h(w_h, w_h) \geq \constsipcoerc\norm{w_h}{\operatorname{sip}}^2.
  \end{equation}
\end{lemma}

The bilinear form $b_h(\cdot, \cdot)$ is an upwinding discretisation of the advective term, where the classical upwinding method is recovered when $\mu = 1$. A natural notion of error when considering the problem (\ref{eq:dg-disc}) is the energy norm defined by
\begin{equation}
  \enorm{w_h}^2
  :=
  \norm{w_h}{\Ltwo{\domain}}^2
  +
  \timestep\norm{w_h}{\operatorname{sip}}^2
  +
  \frac{\timestep\upen}{2}\norm{\lvert\nabla\elpot\cdot\normal\rvert^{1/2}\ujump{w_h}}{\Ltwo{\mathcal{E}_{\rm{int}}}}^2,
\end{equation}
and the well-posedness of the discrete problem is a consequence of the
following coercivity result for the bilinear form $\cA_h(\cdot,
\cdot)$, the proof of which follows the same lines as \cite[Lemma 4.59]{DiPietro2012}.

\begin{lemma}[{Discrete coercivity}]
  \label{lem:dg-coerc}
  Let $w_h\in\finspace_p$ and $\elpot\in\Sob{2}{\infty}{\domain}$. With $\constsipcoerc$ and $\constSob^K$ defined as in Lemma \ref{lem:sip-coerc}. If $\Delta\elpot \leq 0$, then $\cA_h(\cdot, \cdot)$ is coercive on $\finspace_p$, with
  \begin{equation}
    \label{eq:dg-coerc-one}
    \cA_h(w_h, w_h) \geq \constcoerc\enorm{w_h}^2,\quad\text{ where }\constcoerc := \min\!\left(1, \constsipcoerc\right).
  \end{equation}
  Otherwise, if $\Delta\elpot \nleq 0$, then $\cA_h(\cdot, \cdot)$ is coercive on $\finspace_p$ provided $\timestep < 2\norm{\Delta\elpot}{\L{\infty}{\domain}}^{-1}$, we have
  \begin{equation}
    \label{eq:dg-coerc-two}
    \cA_h(w_h, w_h) \geq \constcoerc\enorm{w_h}^2,\quad\text{ where }\constcoerc := \min\!\left(1 - \frac{\timestep}{2}\norm{\Delta\elpot}{\L{\infty}{\domain}}, \constsipcoerc\right).
    \end{equation}
  \end{lemma}
  
\begin{corollary}
  Let $u_0\in\Ltwo{\domain}$, and let the assumptions of Lemma \ref{lem:dg-coerc} be satisfied. Then there exists a unique $u_h^n\in\fes_p$ solving (\ref{eq:dg-disc}), for each $n = 1, 2, \ldots, N$.
\end{corollary}

The discrete solution to (\ref{eq:dg-disc}) enjoys the satisfaction of the following stability result.
\begin{lemma}[Discrete stability]
  \label{lem:dg-disc-energy}
  Let the conditions of Lemma \ref{lem:dg-coerc} be satisfied. For $n = 1, 2, \ldots, N$, let $\finddvar^n\in\finspace_p$ be the solution to (\ref{eq:dg-disc}), with $0 \leq u_0\in\L{\infty}{\domain}$. Then
  \begin{equation}
    \begin{aligned}
      \half\norm{u_h^n}{\Ltwo{\domain}}^2 \leq &\half\norm{u_h^{n-1}}{\Ltwo{\domain}}^2 - \constsipcoerc\timestep\norm{u_h^n}{\operatorname{sip}}^2 - \frac{\timestep\upen}{2}\norm{\lvert\nabla\elpot\cdot\normal\rvert^{1/2}\ujump{u_h^n}}{\Ltwo{\mathcal{E}_{\rm{int}}}}^2\\
                                               &- \timestep\sum_{K\in\mathcal T}\qp{-\half\int_K\qp{u_h^n}^2\Delta\elpot \d\point} ,
    \end{aligned}
  \end{equation}
  and the sequence $\norm{u_h^n}{\Ltwo{\domain}}^2$ is monotonically decreasing if $\Delta\elpot \leq 0$.
\end{lemma}
\begin{proof}
  Integrating (\ref{eq:upwind-def}) by parts, for all $w_h\in\finspace_p$, we find that
  \begin{equation}
    \label{eq:upwind-coerc}
    b_h(w_h, w_h) = \sum_{K\in\mathcal T}\qp{-\half\int_Kw_h^2\Delta\elpot \d\point} + \frac{\upen}{2}\norm{\lvert\nabla\elpot\cdot\normal\rvert^{1/2}\ujump{w_h}}{\Ltwo{\mathcal{E}_{\rm{int}}}}^2.
  \end{equation}
  Choosing $\finddtest = u_h$ in (\ref{eq:dg-disc}), we obtain
  \begin{equation}
    \cA_h(\finddvar^n, \finddvar^n) = \ip{\finddvar^{n-1}}{\finddvar^n}.
  \end{equation}
  and then using Lemma \ref{lem:sip-coerc}, the Cauchy-Schwarz and Young inequalities, and (\ref{eq:upwind-coerc}) completes the proof.
\end{proof}

\begin{remark}[Non-conforming DMP]
  \label{rem:M-matrix}
  In the conforming setting, at least with $\Delta \psi \leq 0$, we
  can use lumping to ensure the discrete problem forms an M-matrix at
  the algebraic level \cite{Thomee2006Galerkin} which guarantees a
  discrete maximum principle. In the non-conforming setting things are
  less clear due to the coupling between elements through the jump
  terms. It is unclear whether a DMP for a SIP discretisation can be
  proven for a simple diffusion equation, see \cite[\S
    9.4]{barrenechea2024finite} for a further discussion and references.
\end{remark}

\subsection{A priori error analysis}
\label{sec:dg-a-priori}

In this section we show an a priori error bound in the energy norm for
the discretisation (\ref{eq:dg-disc}). The error analysis is based on
the following lemmata, see also \cite{dong2020recovered}, where we
make use of the stronger norm
\begin{equation}
  \label{eq:star-norm}
  \enorm{w_h}_*^2 := \enorm{w_h}^2 + \timestep\norm{\lvert\nabla\elpot\cdot\normal\rvert^{1/2}\avg{w_h}}{\Ltwo{\mathcal{E}_{\rm{int}}}}^2.
\end{equation}

\begin{lemma}[{Boundedness of $\cA_h(\cdot, \cdot)$}]
  \label{lem:discrete-boundedness}
  Let $w_h, v_h\in\H{2}{\mathcal{T}}$. There exists $\constbound > 0$, such that
  \begin{equation}
    \cA_h(w_h, v_h) \leq \constbound\enorm{w_h}_*\enorm{v_h}.
  \end{equation}
\end{lemma}

\begin{lemma}[Consistency of $\cA_h(\cdot, \cdot)$]
  \label{lem:Galerkin-orthog}
  Let $u^n\in\H{2}{\domain}\cap\Honezero{\domain}$ solve (\ref{eq:dd-semi}), and let $u_h^n\in\finspace_p$ solve (\ref{eq:dg-disc}). Then
  \begin{equation}
    \cA_h(u^n - u_h^n, v_h) = \ip{u^{n-1} - u_h^{n-1}}{v_h} \Foreach v_h\in\finspace_p.
  \end{equation}
\end{lemma}
\begin{proof}
  Let $A : \H{2}{\domain}\longrightarrow\Ltwo{\domain}$ be the operator associated with the bilinear form $\cA(\cdot, \cdot)$. That is, for $w\in\H{2}{\domain}$,
  \begin{equation}
    \ip{Aw}{v} = \cA(w, v) \Foreach v\in\Honezero{\domain}.
  \end{equation}
  The strong form of the temporal semi-discretisation (\ref{eq:dd-semi}) then reads: for $u^0 := u_0\in\L{\infty}{\domain}$, with $u_0 \geq 0$, and $n = 1, 2, \ldots, N$, seek $u^n\in\H{2}{\domain}\cap\Honezero{\domain}$ such that
  \begin{equation}
    \label{eq:dd-semi-strong}
    \ip{Au^n}{v} = \ip{u^{n-1}}{v} \Foreach v\in\Ltwo{\domain}.
  \end{equation}
  Also, let $A_h : \finspace_p\longrightarrow\finspace_p$ be the operator associated with the bilinear form $\cA_h(\cdot, \cdot)$, i.e. for $w_h\in\finspace_p$,
  \begin{equation}
    \ip{A_hw_h}{v_h} = \cA_h(w_h, v_h) \Foreach v_h\in\finspace_p.
  \end{equation}
  As a result of the consistency of the discrete operator $A_h$, and using (\ref{eq:dd-semi-strong}), we then have
  \begin{equation}
    \label{eq:bilinear-form-consistency}
    \cA_h(u^n, v_h) = \ip{A_hu^n}{v_h} = \ip{Au^n}{v_h} = \ip{u^{n-1}}{v_h} \Foreach v_h\in\finspace_p.
  \end{equation}
  The result then follows by taking the difference of equations (\ref{eq:bilinear-form-consistency}) and (\ref{eq:dg-disc}).
\end{proof}

\begin{lemma}[A Strang-type lemma]
  \label{lem:Strang-type-lemma}
  For $n = 1, 2,\ldots, N$, let $u^n\in\H{2}{\domain}\cap\Honezero{\domain}$ solve (\ref{eq:dd-semi}) and $u_h^n\in\finspace_p$ solve (\ref{eq:dg-disc}). Also, let $\constcoerc$ and $\constbound$ be defined as in Lemma \ref{lem:dg-coerc} and Lemma \ref{lem:discrete-boundedness}, respectively. Under the assumptions of Lemma \ref{lem:dg-coerc}, we have
  \begin{equation}
    \enorm{u^n - u_h^n} \leq \left(1 + \frac{\constbound}{\constcoerc}\right)\inf_{w_h\in\finspace_p}\enorm{u^n - w_h}_*
    +
    \frac{1}{\constcoerc}
    \sup_{v_h \in \finspace_p}
    \frac{\ip{u^{n-1} - u_h^{n-1}}{v_h}}{\enorm{v_h}}.
  \end{equation}
\end{lemma}
\begin{proof}
  Let $w_h\in\finspace_p$. The triangle inequality gives
  \begin{equation}
    \label{eq:Strang-triangle-ineq}
    \enorm{u^n - u_h^n} \leq \enorm{u^n - w_h} + \enorm{w_h - u_h^n},
  \end{equation}
  and we now control the second term on the right-hand side. Using first Lemma \ref{lem:dg-coerc} and then Lemma \ref{lem:Galerkin-orthog}, it follows that
  \begin{align}
    \constcoerc\enorm{w_h - u_h^n}^2 &\leq \cA_h(w_h - u_h^n, w_h - u_h^n) \\
    &= \cA_h(w_h - u^n, w_h - u_h^n) + \cA_h(u^n - u_h^n, w_h - u_h^n) \\
    &= \cA_h(w_h - u^n, w_h - u_h^n) + \ip{u^{n-1} - u_h^{n-1}}{w_h - u_h^n}.
  \end{align}
  Lemma \ref{lem:discrete-boundedness} then gives
  \begin{align}
    \constcoerc\enorm{w_h - u_h^n}^2
    &\leq
    \enorm{w_h - u_h^n}
    \qp{
      \constbound\enorm{w_h - u^n}_*
      +
      \frac{\ip{u^{n-1} - u_h^{n-1}}{w_h - u_h^n}}
           {\enorm{w_h - u_h^n}}
    }
    \\
    &\leq
    \enorm{w_h - u_h^n}
    \left(\constbound\enorm{w_h - u^n}_*
    +
    \sup_{v_h \in \finspace_p}
    \frac{\ip{u^{n-1} - u_h^{n-1}}{v_h}}{\enorm{v_h}}
    \right),
  \end{align}
  and dividing by $\constcoerc\enorm{w_h - u_h^n}$ results in
  \begin{equation}
    \enorm{w_h - u_h^n}
    \leq
    \frac{\constbound}{\constcoerc}\enorm{w_h - u^n}_*
    +
    \frac{1}{\constcoerc}
    \sup_{v_h \in \finspace_p}
    \frac{\ip{u^{n-1} - u_h^{n-1}}{v_h}}{\enorm{v_h}}.
  \end{equation}
  Returning to (\ref{eq:Strang-triangle-ineq}), we thus have
  \begin{equation}
    \enorm{u^n - u_h^n} \leq \left(1 + \frac{\constbound}{\constcoerc}\right)\enorm{u^n - w_h}_*
    +
    \frac{1}{\constcoerc}
    \sup_{v_h \in \finspace_p}
    \frac{\ip{u^{n-1} - u_h^{n-1}}{v_h}}{\enorm{v_h}},
  \end{equation}
  and choosing $w_h\in\finspace_p$ to minimise $\enorm{u^n - w_h}_*$ completes the proof.
\end{proof}

\section{A bound-preserving method}
\label{sec:nodal-dg-disc}

As highlighted in Remark \ref{rem:M-matrix}, there is no guarantee that the solution to the discontinuous Galerkin discretisation (\ref{eq:dg-disc}) respects a discrete maximum principle, like the semi-discrete problem was shown to in Lemma \ref{lem:semi-discrete-maximum-principle}. In this section we adopt the approach first introduced for reaction-diffusion
equations in \cite{Barrenechea2023}, and extended to
reaction-advection-diffusion equations in \cite{Amiri2024}, whereby we
seek solutions that respect the bounds at, but not necessarily between,
element nodes. Figure \ref{fig:nodal-demo} demonstrates the concept
for piecewise linear and piecewise quadratic polynomials in one
dimension. We emphasise that for $p=1$ the solution satisfies the
bounds \emph{globally}, and not just at the nodes.
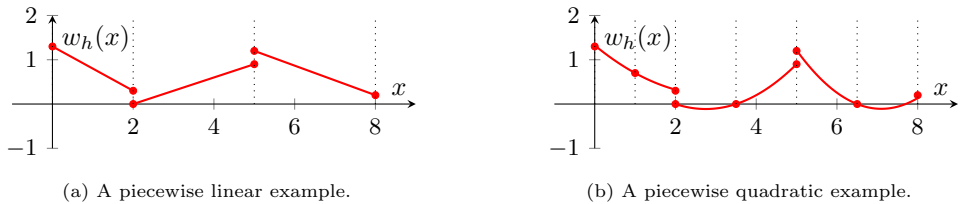
\begin{figure}[h]
  \centering
  \subcaptionbox{A piecewise linear example.}
  {\begin{tikzpicture}
  \begin{axis}[
    scale = 0.4,
    axis lines = middle,
    xlabel = {$x$},
    ylabel = {$w_h(x)$},
    tick label style={font=\small},
    xmin = -1, xmax = 9,
    ymin = -1, ymax = 2,
    domain = -1:9,
    width = 15cm,
    height = 6cm,
    clip = false,
    legend pos = outer north east,
    samples = 100,
]

\addplot [
    domain=0:2,
    thick,
    red,
    samples=100
] {x*(0.3-1.3)/(2-0)+1.3};

\addplot [
    domain=2:5,
    thick,
    red,
    samples=100
] {(x-2)*(0.9)/(5-2)};

\addplot [
    domain=5:8,
    thick,
    red,
    samples=100
] {(x-5)*(0.2-1.2)/(8-5) + 1.2};

\filldraw[red] (0,1.3) circle (1.3pt); 
\filldraw[red] (2,0.3) circle (1.3pt); 

\filldraw[red] (2,0) circle (1.3pt); 
\filldraw[red] (5,0.9) circle (1.3pt); 

\filldraw[red] (5,1.2) circle (1.3pt); 
\filldraw[red] (8,0.2) circle (1.3pt); 

\draw[dotted] (0,0) -- (0,2);
\draw[dotted] (2,0) -- (2,2);
\draw[dotted] (5,0) -- (5,2);
\draw[dotted] (8,0) -- (8,2);

\end{axis}
\end{tikzpicture}

}%
  \hspace{4em}
  \subcaptionbox{A piecewise quadratic example.}
  {\begin{tikzpicture}
  \begin{axis}[
    scale = 0.4,
    axis lines = middle,
    xlabel = {$x$},
    ylabel = {$w_h(x)$},
    tick label style={font=\small},
    xmin = -1, xmax = 9,
    ymin = -1, ymax = 2,
    domain = -1:9,
    width = 15cm,
    height = 6cm,
    clip = false,
    legend pos = outer north east,
    samples = 100,
]

\addplot [
    domain=0:2,
    thick,
    red,
    samples=100
] {0.5*(-0.5*x+1.5)^2+0.2};

\addplot [
    domain=2:5,
    thick,
    red,
    samples=100
] {0.2*(x-2)*(x-3.5)};

\addplot [
    domain=5:8,
    thick,
    red,
    samples=100
] {0.3*(x-6.5)*(x-7.7)};

\filldraw[red] (0,1.3) circle (1.3pt); 
\filldraw[red] (1,0.7) circle (1.3pt); 
\filldraw[red] (2,0.3) circle (1.3pt); 

\filldraw[red] (2,0) circle (1.3pt); 
\filldraw[red] (3.5,0) circle (1.3pt); 
\filldraw[red] (5,0.9) circle (1.3pt); 

\filldraw[red] (5,1.2) circle (1.3pt); 
\filldraw[red] (6.5,0) circle (1.3pt); 
\filldraw[red] (8,0.2) circle (1.3pt); 

\draw[dotted] (0,0) -- (0,2);
\draw[dotted] (1,0) -- (1,2);
\draw[dotted] (2,0) -- (2,2);
\draw[dotted] (3.5,0) -- (3.5,2);
\draw[dotted] (5,0) -- (5,2);
\draw[dotted] (6.5,0) -- (6.5,2);
\draw[dotted] (8,0) -- (8,2);

\end{axis}
\end{tikzpicture}

}
  \caption{Examples of piecewise polynomial functions, on a non-uniform one-dimensional mesh, which are nodally non-negative. The dotted lines denote the locations of the Lagrange nodes of the elements. For piecewise linear functions the non-negativity is global, but for higher-order polynomials the function may go negative between the nodes.}
  \label{fig:nodal-demo}
\end{figure}

Let us begin by defining the closed convex subset $\finspace_p^+\subseteq\finspace_p$ by restricting finite element functions at the Lagrange nodes of each element, denoted $\mathcal{L}_K$ for all $K\in\triang$, to be between zero and the supremal value of the initial condition:
\begin{equation}
  \finspace_p^+ := \left\{w_h\in\finspace_p : 0 \leq w_h\!\left(\node_i\right) \leq \norm{u_0}{\L{\infty}{\domain}},\text{ for all }\node_i\in\mathcal{L}_K,\text{ for all }K\in\triang\right\}.
\end{equation}
\begin{remark}
  The use of the upper limit $\norm{u_0}{\Linf{\domain}}$ is rather coarse and does not provide control as the solution decays towards zero. Given that the problem at each time step satisfies an elliptic maximum principle, one could instead consider a family of convex sets where the upper limit is given by $\norm{u^{n-1}}{\Linf{\domain}}$, providing sharper bounds as the solution evolves. At the discrete level $\norm{u_h^{n-1}}{\Linf{\domain}}$ could be used, with implementations requiring only an upper limit update step to be added in the solving routine. Time-dependence in the definition of the convex set is also considered in the companion paper \cite{Amiri2025nodally}.
\end{remark}

We then pose the following method, where we write $\qp{u_h^n}^+\in\finspace_p^+$ to distinguish the nodally bound-preserving (or \emph{constrained}) solution from the \emph{unconstrained} solution $u_h^n\in\finspace_p$ to (\ref{eq:dg-disc}).

Given $\elpot\in\Sob{2}{\infty}{\domain}$ and initial data $0 \leq \ddvarinit\in\L{\infty}{\domain}$, so $\finddvar^0 = \interp{\ddvarinit}\in\finspace_p$, seek $\qp{\finddvar^n}^+\in\finspace_p^+$, for $n = 1, 2, \ldots, N$, such that
\begin{equation}
  \label{eq:drift-diffusion-full-disc}
  \cA_h\qp{\qp{\finddvar^n}^+, \finddtest - \qp{\finddvar^n}^+} \geq \ip{\qp{\finddvar^{n-1}}^+}{\finddtest - \qp{\finddvar^n}^+}\quad\forall\finddtest\in\finspace_p^+.
\end{equation}

\begin{remark}
  A discrete approximation to the initial condition in $\finspace_p^+$ can be obtained using a bound-preserving operator $\Pi^+_h : \Ltwo{\domain}\longrightarrow\finspace_p^+$.  If $\ddvarinit$ is smooth enough to admit point values we may use the Lagrange interpolant. If $\ddvarinit$ is less smooth then other constructions are available (see e.g. \cite{ashby2024duality, Chen2000, Veeser2019}).
\end{remark}

\begin{lemma}
  Let $\elpot\in\Sob{2}{\infty}{\domain}$ and $0 \leq \ddvarinit\in\L{\infty}{\domain}$. Suppose the conditions of Lemma \ref{lem:dg-coerc} are satisfied. Then, for $n = 1, 2, \ldots, N$, there exists a unique $\qp{\finddvar^n}^+\in\finspace_p^+$ that solves (\ref{eq:drift-diffusion-full-disc}).
\end{lemma}
\begin{proof}
  By Lemma \ref{lem:dg-coerc}, the bilinear form $\cA_h(\cdot, \cdot)$ is coercive on $\finspace_p$ under the stated assumptions. The result then follows by applying Stampacchia's Theorem (\cite[pg. 24, Theorem 2.1]{Kinderlehrer1980}).
\end{proof}

By design, the solution of (\ref{eq:drift-diffusion-full-disc}) is
nodally bound-preserving. We now see that it also satisfies a discrete energy decay, which is similar to Lemma \ref{lem:dg-disc-energy}.

\begin{lemma}[Discrete stability]
  \label{lem:drift-diffusion-full-disc-energy}
  Let the conditions of Lemma \ref{lem:dg-coerc} be satisfied. For $n = 1, 2, \ldots, N$, let $\qp{\finddvar^n}^+\in\finspace_p^+$ be the solution to (\ref{eq:drift-diffusion-full-disc}), with $0 \leq u_0\in\L{\infty}{\domain}$. Then
  \begin{equation}
    \begin{aligned}
      \half\norm{\qp{u_h^n}^+}{\Ltwo{\domain}}^2 \leq &\half\norm{\qp{u_h^{n-1}}^+}{\Ltwo{\domain}}^2 - \timestep\constsipcoerc\norm{\qp{u_h^n}^+}{\operatorname{sip}}^2 - \frac{\timestep\upen}{2}\norm{\lvert\nabla\elpot\cdot\normal\rvert^{1/2}\ujump{\qp{u_h^n}^+}}{\Ltwo{\mathcal{E}_{\rm{int}}}}^2\\
                                               &- \timestep\sum_{K\in\mathcal T}\qp{-\half\int_K\qp{\qp{u_h^n}^+}^2\Delta\elpot \d\point},
    \end{aligned}
  \end{equation}
  and the sequence $\norm{\qp{u_h^n}^+}{\Ltwo{\domain}}^2$ is monotonically decreasing if $\Delta\elpot \leq 0$.
\end{lemma}
\begin{proof}
  Testing with $\finddtest = 0$ in (\ref{eq:drift-diffusion-full-disc}), we obtain
  \begin{equation}
    \cA_h\qp{\qp{\finddvar^n}^+, \qp{\finddvar^n}^+} \leq \ip{\qp{\finddvar^{n-1}}^+}{\qp{\finddvar^n}^+},
  \end{equation}
  and the result then follows the same argument as Lemma \ref{lem:dg-disc-energy}.
\end{proof}

\subsection{Higher-order time discretisations}

Higher-order time discretisations can also be posed whilst retaining the structure preservation properties. For example, an unconstrained second-order Crank-Nicolson method would read as follows. Given $\elpot\in\Sob{2}{\infty}{\domain}$ and $0 \leq \ddvarinit\in\L{\infty}{\domain}$, so $\finddvar^0 = \interp{\ddvarinit}\in\finspace_p$, seek $\finddvar^n\in\finspace_p$, for $n = 1, 2, \ldots, N$, such that
\begin{equation}
  \label{eq:Crank-Nicolson-disc}
  \ip{u_h^n}{v_h} + \timestep a_h\left(\frac{u_h^n + u_h^{n-1}}{2}, v_h\right) + \timestep b_h\left(\frac{u_h^n + u_h^{n-1}}{2}, v_h\right) = \ip{u_h^{n-1}}{v_h} \Foreach v_h\in\finspace_p.
\end{equation}

Then the following stability result, which is an analogue of Lemma \ref{lem:dg-disc-energy} holds.
\begin{lemma}
  Let the conditions of Lemma \ref{lem:dg-coerc} be satisfied, and let $\finddvar^n\in\finspace_p$ solve (\ref{eq:Crank-Nicolson-disc}), for $n = 1, 2, \ldots, N$. Then
  \begin{equation}
    \begin{aligned}
      \half\norm{u_h^n}{\Ltwo{\domain}}^2 \leq &\half\norm{u_h^{n-1}}{\Ltwo{\domain}}^2 - \timestep\constsipcoerc\norm{\frac{u_h^n + u_h^{n-1}}{2}}{\operatorname{sip}}^2 - \frac{\timestep\upen}{2}\norm{\lvert\nabla\elpot\cdot\normal\rvert^{1/2}\ujump{\frac{u_h^n + u_h^{n-1}}{2}}}{\Ltwo{\mathcal{E}_{\rm{int}}}}^2\\
                                               &- \timestep\sum_{K\in\mathcal T}\qp{-\half\int_K\left(\frac{u_h^n + u_h^{n-1}}{2}\right)^2\Delta\elpot \d\point},
    \end{aligned}
  \end{equation}
  which is monotonically decreasing if $\Delta\elpot \leq 0$.
\end{lemma}
\begin{proof}
  Choosing $v_h = \half\left(u_h^n + u_h^{n-1}\right)$ in equation (\ref{eq:Crank-Nicolson-disc}) gives
  \begin{equation}
    \half\norm{u_h^n}{\Ltwo{\domain}}^2 = \half\norm{u_h^{n-1}}{\Ltwo{\domain}}^2 - \timestep a_h\left(\frac{u_h^n + u_h^{n-1}}{2}, \frac{u_h^n + u_h^{n-1}}{2}\right) - \timestep b_h\left(\frac{u_h^n + u_h^{n-1}}{2}, \frac{u_h^n + u_h^{n-1}}{2}\right),
  \end{equation}
  and the result then follows by Lemma \ref{lem:sip-coerc} and (\ref{eq:upwind-coerc}).
\end{proof}

A corresponding nodally bound-preserving Crank-Nicolson discretisation can be posed and the related analogue to Lemma \ref{lem:drift-diffusion-full-disc-energy} proceeds in a similar fashion to the backward Euler case.

\begin{remark}
  The methods presented herein can be combined with other higher-order dissipative-type time discretisations, such as dG-in-time approaches \cite{Thomee2006Galerkin}, backward differentiation formulae (BDF) \cite{Butcher2008numerical}, and implicit-explicit additive Runge--Kutta (ImEx ARK) methods \cite{Kennedy2003additive}. In each case, it should be possible to guarantee structure preservation.
\end{remark}

\subsection{A priori error analysis}

The analysis for the nodally bound-preserving method follows from that of Section \ref{sec:dg-a-priori} without much modification. The main difference lies in the following lemma.
\begin{lemma}
  \label{lem:nodal-Galerkin-orthog}
  For $n = 1, 2,\ldots,N$, let $u^n\in\H{2}{\domain}\cap\Honezero{\domain}$ solve (\ref{eq:dd-semi}) and $\qp{u_h^n}^+\in\finspace_p^+$ solve (\ref{eq:drift-diffusion-full-disc}). Then
  \begin{equation}
    \cA_h\qp{u^n - \qp{u_h^n}^+, v_h - \qp{u_h^n}^+} \leq \ip{u^{n-1} - \qp{u_h^{n-1}}^+}{v_h - \qp{u_h^n}^+} \Foreach v_h\in\finspace_p^+.
  \end{equation}
\end{lemma}
\begin{proof}
  From (\ref{eq:bilinear-form-consistency}) it follows that
  \begin{equation}
    \cA_h\qp{u^n, v_h - \qp{u_h^n}^+} = \ip{u^{n-1}}{v_h - \qp{u_h^n}^+} \Foreach v_h\in\finspace_p^+,
  \end{equation}
  and (\ref{eq:drift-diffusion-full-disc}) implies that
  \begin{equation}
    -\cA_h\qp{\qp{u_h^n}^+, v_h - \qp{u_h^n}^+} \leq -\ip{\qp{u_h^{n-1}}^+}{v_h - \qp{u_h^n}^+} \Foreach v_h\in\finspace_p^+.
  \end{equation}
  Combining these yields the result.
\end{proof}

The proof of the following corollary now proceeds in the same way as that of Lemma \ref{lem:Strang-type-lemma}, with Lemma \ref{lem:nodal-Galerkin-orthog} replacing Lemma \ref{lem:Galerkin-orthog} as appropriate.
\begin{corollary}
  \label{cor:nodal-Strang-type-lemma}
  For $n = 1, 2,\ldots,N$, let $u^n\in\H{2}{\domain}\cap\Honezero{\domain}$ solve (\ref{eq:dd-semi}), let $\qp{u_h^n}^+\in\finspace_p^+$ solve (\ref{eq:drift-diffusion-full-disc}), and let $\constcoerc$ and $\constbound$ be defined as in Lemma \ref{lem:dg-coerc} and Lemma \ref{lem:discrete-boundedness}, respectively. Under the assumptions of Lemma \ref{lem:dg-coerc}, we have
  \begin{equation}
    \enorm{u^n - \qp{u_h^n}^+} \leq \left(1 + \frac{\constbound}{\constcoerc}\right)\inf_{w_h\in\finspace_p^+}\enorm{u^n - w_h}_* + \frac{1}{\constcoerc}\sup_{v_h\in\finspace_p}\frac{\ip{u^{n-1} - \qp{u_h^{n-1}}^+}{v_h}}{\enorm{v_h}}.
  \end{equation}
\end{corollary}

\section{Numerical Experiments}
\label{sec:numerics}

The nodally bound-preserving scheme (\ref{eq:drift-diffusion-full-disc}) was implemented in the FEniCS (Legacy) software \cite{Logg2010, Logg2012}, with the variational inequality treated at each time step via the classical extragradient method, which we now describe \cite{Korpelevich1976}. This simple-to-implement method is by no means the only way to solve variational inequalities; for further details, we point to the book \cite{Facchinei2003finite-dimensional} and the recent review paper \cite{Beznosikov2023smooth}, which is focused towards stochastic approaches but also presents deterministic algorithms. In particular, we highlight the recent work \cite{Ashby2025nodally}, where an active set method is used to solve the variational inequality arising from the nodally bound-preserving method applied to hyperbolic advection-reaction problems. Although that study concerns a continuous Galerkin method, active set approaches have also found success with dG methods \cite{Wang2010discontinuous, Gudi2014aposteriori}.

Let $\vec A$ and $\vec L$ be the assembled finite element stiffness matrix and load vector, respectively, let $\gamma > 0$, and let $\mathcal{P}$ be the realisation of a projection of $u_h\in\finspace_p$ into $\finspace_p^+$, which we define, for all $K\in\triang$ and $\node_i\in\mathcal{L}_K$, by
\begin{equation}
  \mathcal{P}\qp{u_h\qp{\node_i}} := \min\qp{\norm{u_0}{\L{\infty}{\domain}}, \max\qp{0, u_h\qp{\node_i}}}.
\end{equation}
In Section \ref{ex:drift-diffusion-structure-preservation-initial-conditions} we observe the two-sided bound, however, in the other examples we remove the upper limit, since only positivity of the solutions can be shown. Taking $\vec u^0$ to be the solution vector obtained by solving (\ref{eq:dg-disc}), the iteration then proceeds as
\begin{align}
  \label{eq:iteration}
  \vec v^m &= \mathcal{P}\left(\vec u^{m-1} - \gamma\left(\vec A\vec u^{m-1} - \vec L\right)\right), \\
  \vec u^m &= \mathcal{P}\left(\vec u^{m-1} - \gamma\left(\vec A\vec v^m - \vec L\right)\right),
\end{align}
and the process is terminated when $\norm{\vec u^m - \vec u^{m-1}}{\ell^2} < \operatorname{tol} := 10^{-6}$. Other norms instead of $\ell^2$ could be used, however, in our implementation there seems to be no significant difference between the results using the $\ell^2$, $\ell^{\infty}$, and $\Ltwo{\domain}$ norms.

The experiments are all carried out in two spatial dimensions on triangular meshes, and unless otherwise stated, we use piecewise linear elements, a direct linear solver, and take $\pen = 10$, $\mu = 1$, and $\gamma = 10^{-5}$.

\subsection{Convergence on a uniform mesh with a smooth solution}
\label{ex:drift-diffusion-benchmark}

We begin by verifying the convergence properties of the discretisation (\ref{eq:drift-diffusion-full-disc}) on a uniform triangular mesh with upper right diagonals of the domain $\domain\times\timedom = \left(0, 1\right)^2\times\left(0, 1/2\right]$, using polynomial degrees $p = 1, 2$. To obtain a manufactured solution we choose
\begin{align}
  \ddvar &= \sin{\pi t}\sin{\pi x}\sin{\pi y}, \\
  \elpot &= \sin{\pi t}\cos{\pi x}\cos{\pi y},
\end{align}
and include an additional appropriate forcing term in (\ref{eq:drift-diffusion-full-disc}). The initial condition is thus $\ddvarinit = 0$. The error $\enorm{\ddvar^N - \qp{\finddvar^N}^+}$ is examined on a sequence of successively finer meshes, with mesh size parameters chosen such that $h^{-1} = 4, 8, 16, 32, 64$, and we fix $\timestep = h^{p+1}$. The results are depicted in Figure \ref{fig:drift-diffusion-benchmark}, and demonstrate $\mathcal{O}(h^{p+1})$ convergence. In this case the iteration given by (\ref{eq:iteration}) converged in a single step at each time step, meaning the solution to the dG scheme (\ref{eq:dg-disc}) satisfied the bounds at the nodes without the need for projection.

\begin{figure}[h]
  \centering
  \pgfplotstableread[col sep=comma]{
  h,u error,Time taken
  0.25,0.3010842158188604,0.01051187515258789
  0.125,0.07385120420497472,0.04257822036743164
  0.0625,0.01797168968453009,0.6331150531768799
  0.03125,0.004414689422465761,9.48153042793274
  0.015625,0.0010929925484454747,214.08147621154785
}\errorspone

\pgfplotstableread[col sep=comma]{
  h,u error,Time taken
  0.25,0.037070591940242835,0.018883466720581055
  0.125,0.0033208339254063323,0.25254225730895996
  0.0625,0.0003035655324993977,14.94091248512268
  0.03125,2.8834884099524003e-05,534.3486318588257
  0.015625,2.8813806782401585e-06,16574.799842357635
}\errorsptwo

\begin{tikzpicture}[scale=0.75]
  \begin{axis}[
    cycle list/Dark2,
    thick,
    xmode=log,
    ymode=log,
    xlabel=$h$,
    ylabel=$\enorm{\ddvar^N - \qp{\finddvar^N}^+}$,
    grid=both,
    minor grid style={gray!25},
    major grid style={gray!25},
    legend style={at={(1.0, 0.0)},anchor=south east},
    ]
    \addplot[color=black, mark=triangle] table[x=h, y=u error, col sep=comma]{\errorspone};
    \addlegendentry{$p=1$};
    \addplot[color=black, mark=square] table[x=h, y=u error, col sep=comma]{\errorsptwo};
    \addlegendentry{$p=2$};
    
    \coordinate (start-value-1) at (0.3, 0.6);
    \coordinate (start-value-2) at (0.3, 0.2);
    \coordinate (start-value-3) at (0.3, 0.01);
    \coordinate (end-value-1) at (0.015, 0.03);
    \coordinate (end-value-2) at (0.015, 0.0005);
    \coordinate (end-value-3) at (0.015, 1.25e-6);
    
    \draw[dotted] (start-value-2) -- (end-value-2) node[pos=0.7, below] {$2.0$};
    \draw[dotted] (start-value-3) -- (end-value-3) node[pos=0.7, below] {$3.0$};
  \end{axis}
\end{tikzpicture}

  \caption{Convergence of the solution to (\ref{eq:drift-diffusion-full-disc}) on a uniform mesh in the $\enorm{\cdot}$-norm using piecewise polynomials of degree $p = 1, 2$. Convergence is observed at a rate of $\mathcal{O}(h^{p+1})$.}
  \label{fig:drift-diffusion-benchmark}
\end{figure}
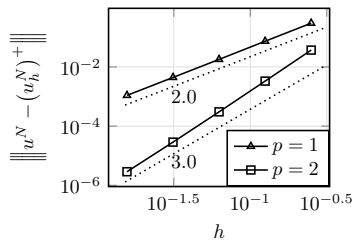

\subsection{Structure preservation with discontinuous initial conditions}
\label{ex:drift-diffusion-structure-preservation-initial-conditions}

We first demonstrate the structure preservation of the method by investigating a problem setup with discontinuous initial conditions in a checkerboard-type pattern. Taking $\domain\times\timedom = \qp{0, 1}^2\times\left(0, 0.03\right]$, we solve (\ref{eq:drift-diffusion-full-disc}) with piecewise linear elements on a uniform triangular mesh with right-aligned diagonals, with mesh size parameter $h = 0.01$ and time step size $\timestep = 3\times10^{-4}$. At initial time the value of $u_0$ is either 0 or 1, as shown in Figure \ref{fig:dd-snapshot-initial}, and we let $\elpot = 100\qp{x + y}$. We note that $\Delta\elpot = 0$.

The evolution of the solution to (\ref{eq:drift-diffusion-full-disc}) is shown through the snapshots in Figure \ref{fig:dd-snapshots}, where the combination of the diffusive and advective behaviours can be seen. A boundary layer is formed towards the bottom left corner of the domain, and Figure \ref{fig:dd-minimum} demonstrates that the the solution to (\ref{eq:dg-disc}) does not remain non-negative in the context of this problem, despite respecting the upper bound, as shown in Figure \ref{fig:dd-maximum}. By design the nodally bound-preserving solution satisfies both constraints. We examine the $L^2$-norm of the solution in Figure \ref{fig:dd-energy}, where monotonic decay is demonstrated for both solutions, in agreement with Lemma \ref{lem:dg-disc-energy} and Lemma \ref{lem:drift-diffusion-full-disc-energy}.

\begin{figure}[h]
  \centering
  \subcaptionbox{$t^n = 0.0$\label{fig:dd-snapshot-initial}}
  [0.19\textwidth]{
    \includegraphics[width=\linewidth]{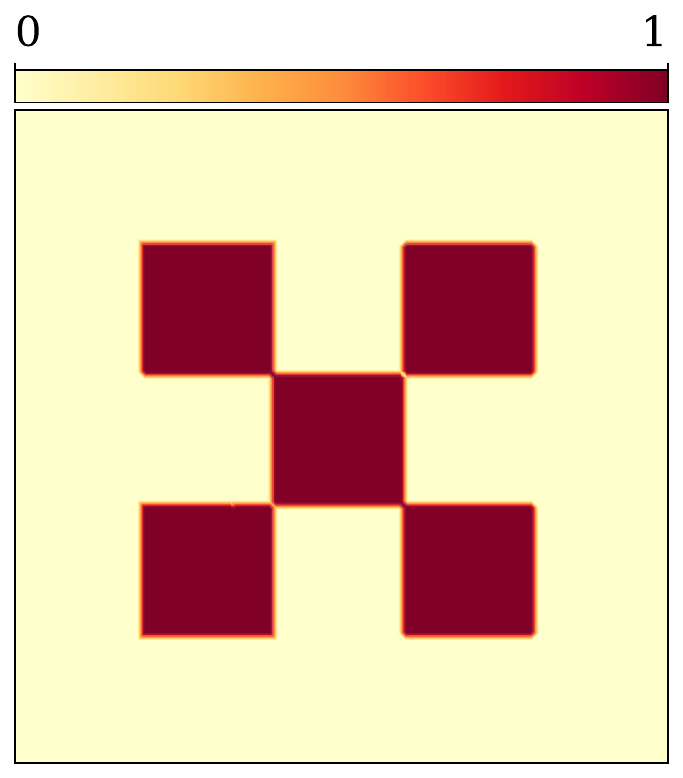}
  }
  \hfill
  \subcaptionbox{$t^n = 9.0\times10^{-4}$}
  [0.19\textwidth]{
    \includegraphics[width=\linewidth]{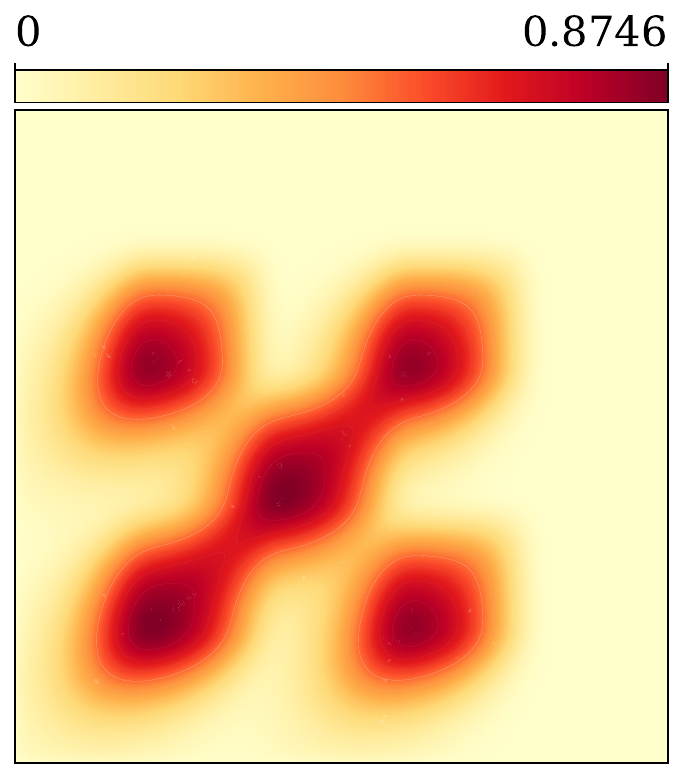}
  }
  \hfill
  \subcaptionbox{$t^n = 2.3\times10^{-3}$}
  [0.19\textwidth]{
    \includegraphics[width=\linewidth]{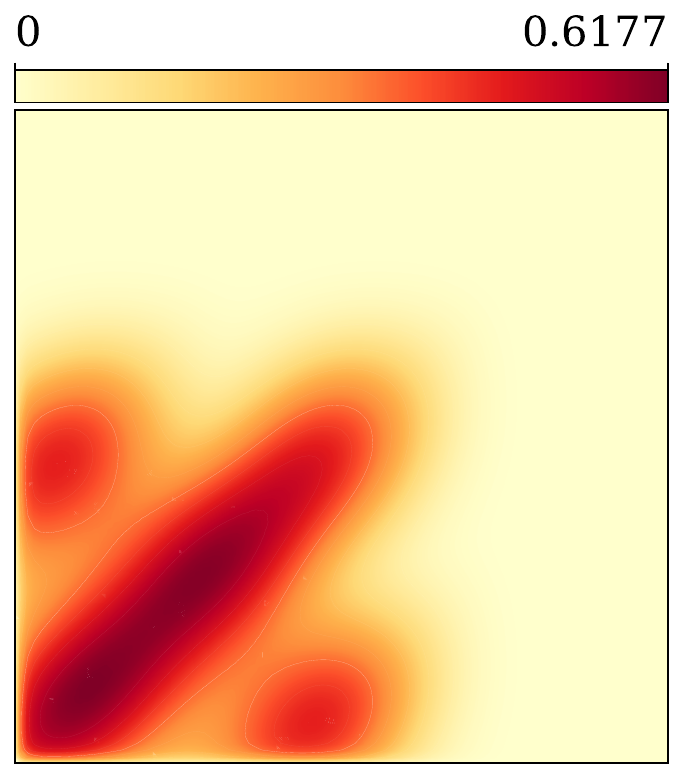}
  }
  \hfill
  \subcaptionbox{$t^n = 6.3\times10^{-3}$}
  [0.19\textwidth]{
    \includegraphics[width=\linewidth]{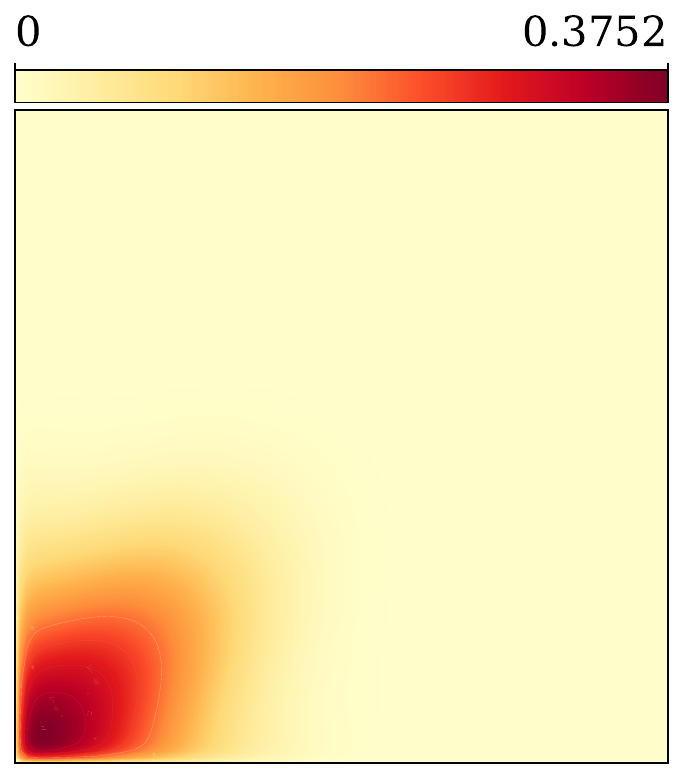}
  }
  \hfill
  \subcaptionbox{$t^n = 1.29\times10^{-2}$}
  [0.19\textwidth]{
    \includegraphics[width=\linewidth]{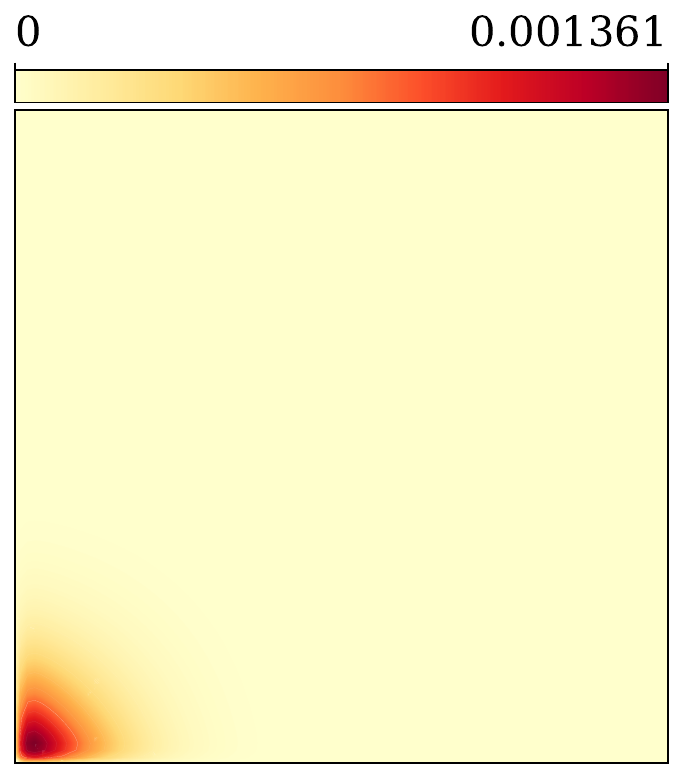}
  }
  \caption{Snapshots of the solution to the example from Section \ref{ex:drift-diffusion-structure-preservation-initial-conditions}.}
  \label{fig:dd-snapshots}
\end{figure}

\begin{figure}[h]
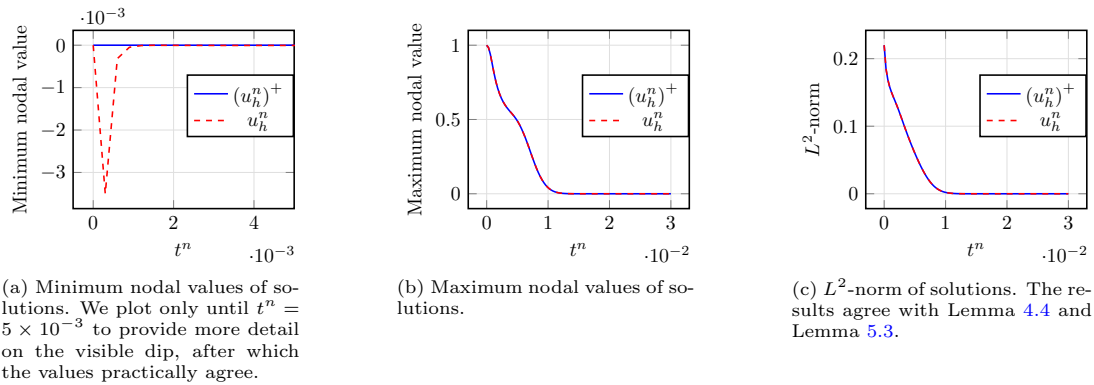

  \centering
  \subcaptionbox{Minimum nodal values of solutions. We plot only until $t^n = 5\times 10^{-3}$ to provide more detail on the visible dip, after which the values practically agree.\label{fig:dd-minimum}}
  {\input{dd_minimum}}
  \hspace{3em}
  \subcaptionbox{Maximum nodal values of solutions.\label{fig:dd-maximum}}
  {\input{dd_maximum}}
  \hspace{3em}
  \subcaptionbox{$L^2$-norm of solutions. The results agree with Lemma \ref{lem:dg-disc-energy} and Lemma \ref{lem:drift-diffusion-full-disc-energy}.\label{fig:dd-energy}}
  {\input{dd_energy}}
  \caption{Plots comparing the nodally bound-preserving and unconstrained solutions for the example in Section \ref{ex:drift-diffusion-structure-preservation-initial-conditions}.}
  \label{fig:dd-minimum-energy}
\end{figure}

\subsection{Structure preservation with time-dependent boundary conditions}
\label{ex:drift-diffusion-structure-preservation}

Next we consider a problem with non-homogeneous boundary conditions. The spatial domain $\qp{0, 1}^2 \setminus B_{0.1}\qp{\qp{1/2, 1/2}}$ is a unit square with a disc of radius $0.1$ removed from the centre, and the temporal domain is the interval $\left(0, 1\right]$. We examine two different choices of $\elpot$, given by $\elpot_1 = 100\sin{\pi \qp{2x - 1/2}}$ and $\elpot_2 = -100\sin{\pi \qp{2x - 1/2}}$, where we observe that $\Delta\elpot_1 \leq 0$ and $\Delta\elpot_2 \nleq 0$ over the considered domain. Defining
\begin{equation}
  g\!\left(t\right) := \half + \half\tanh\!\left(8\left(2t - \half\right)\right),
\end{equation}
we take the Dirichlet boundary conditions $u = 0$ on the boundary of the square and $u = g\!\left(t\right)$ on the boundary of the disc. The initial condition is taken as $u_0 = 0$. We solve the two problems using (\ref{eq:drift-diffusion-full-disc}) with piecewise linear elements on a quasi-uniform Delaunay mesh of the domain that results in $\approx 68,000$ degrees of freedom and a polygonal approximation of the curved boundary. The time step size is fixed at $\timestep = 0.01$, and for this problem we set $\pen = 100$.

Results concerning the solutions of (\ref{eq:dg-disc}) and (\ref{eq:drift-diffusion-full-disc}) with $\psi_1$ are shown in Figures \ref{fig:elpot-one-snapshot} and \ref{fig:elpot-one-layer}, and results with $\psi_2$ are shown in Figures \ref{fig:elpot-two-snapshot} and \ref{fig:elpot-two-layer}.
Figure \ref{fig:elpot-one-snapshot} shows the solution at final time with $\elpot_1$, and Figure \ref{fig:elpot-one-layer} plots the minimum value of the solutions to both (\ref{eq:dg-disc}) and (\ref{eq:drift-diffusion-full-disc}) over time, demonstrating the non-negativity of the latter in contrast to the former. We observe in Figure \ref{fig:elpot-two-snapshot} the formation of boundary layers around the central disc for the problem involving $\elpot_2$. The resulting spurious oscillations are examined in greater detail in Figure \ref{fig:elpot-two-layer}, where it can be seen that the solution to (\ref{eq:dg-disc}) becomes negative, whereas the nodally non-negative solution to (\ref{eq:drift-diffusion-full-disc}) respects the bound. The oscillations also reduce in magnitude.

\begin{figure}[h]
  \centering
  \subcaptionbox{Colour map of the nodally non-negative solution with $\elpot_1$.\label{fig:elpot-one-snapshot}}
  [0.19\textwidth]{\includegraphics[width=\linewidth]{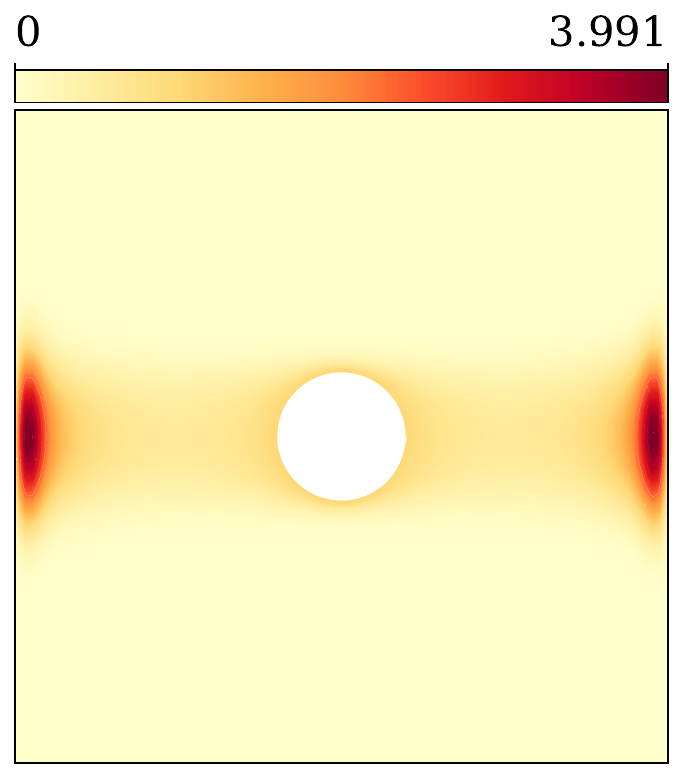}}
  \hfill
  \subcaptionbox{Minimum nodal values of solutions with $\elpot_1$.\label{fig:elpot-one-layer}}
  [0.22\textwidth]{\input{minimum}}
  \hfill
  \subcaptionbox{Colour map of the nodally non-negative solution with $\elpot_2$. The marked line indicates the partial slice shown in Figure \ref{fig:elpot-two-layer}.\label{fig:elpot-two-snapshot}}
  [0.19\textwidth]{
    \includegraphics[width=\linewidth]{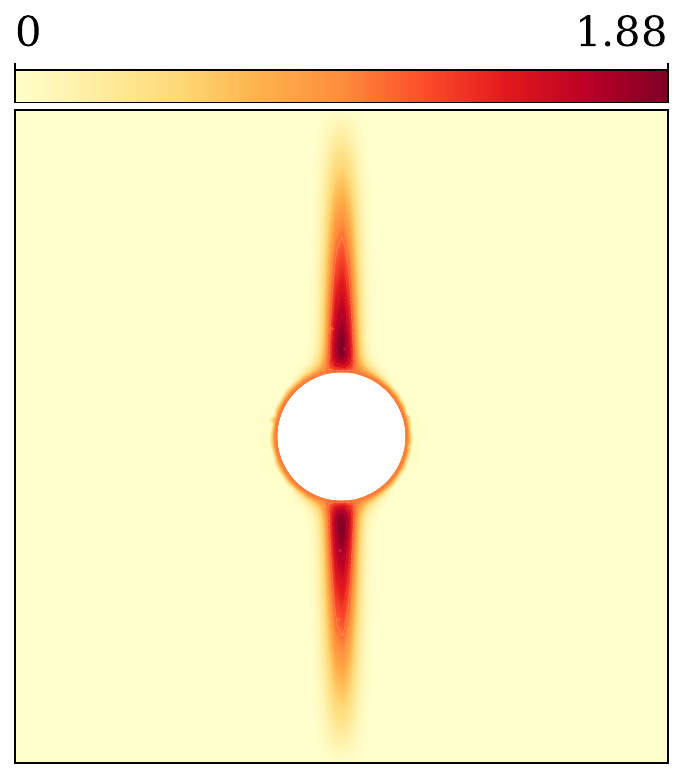}%
    \begin{tikzpicture}[overlay]
      \draw (-2.1, 1.575) -- (-1.85, 1.575);
    \end{tikzpicture}}
  \hfill
  \subcaptionbox{Partial slice, corresponding to the marked line in Figure \ref{fig:elpot-two-snapshot}, along the line $y = 1/2$ with $\elpot_2$.\label{fig:elpot-two-layer}}
  [0.22\textwidth]{\pgfplotstableread[col sep=comma]{
  ,u,Points:0,Points:1,Points:2
158,1.1374e-13,0.3031,0.5,0
157,1.8132e-06,0.31106,0.5,0
161,1.9647e-06,0.31174,0.5,0
159,1.8138e-06,0.31243,0.5,0
137,6.1263e-13,0.32038,0.5,0
116,1.412e-05,0.32837,0.5,0
80,1.5253e-05,0.32903,0.5,0
61,1.4124e-05,0.32968,0.5,0
59,1.5246e-12,0.33767,0.5,0
60,0.00014606,0.34567,0.5,0
147,0.00015737,0.3463,0.5,0
94,0.00014609,0.34693,0.5,0
78,3.5664e-11,0.35494,0.5,0
77,0.0015734,0.36298,0.5,0
79,0.0016734,0.36358,0.5,0
68,0.0015737,0.36418,0.5,0
66,6.3116e-10,0.37222,0.5,0
26,0.02888,0.38053,0.5,0
27,0.029884,0.38104,0.5,0
185,0.029574,0.38134,0.5,0
184,0.00011541,0.39137,0.5,0
4,4.1808e-10,0.39151,0.5,0
5,0.028151,0.39172,0.5,0
154,1.0,0.40012,0.5,0
}\slicenodal

\pgfplotstableread[col sep=comma]{
  ,u,Points:0,Points:1,Points:2
19,-9.8302e-05,0.30501,0.5,0
13,0.0006047,0.31198,0.5,0
12,-0.00083273,0.3237,0.5,0
78,-0.001221,0.32515,0.5,0
88,0.00032738,0.32918,0.5,0
77,0.0026909,0.3371,0.5,0
60,0.0008059,0.34043,0.5,0
41,-0.0053484,0.34893,0.5,0
40,0.0086734,0.36003,0.5,0
18,0.0094184,0.36045,0.5,0
17,0.007172,0.36103,0.5,0
16,-0.026763,0.37162,0.5,0
10,0.032535,0.37726,0.5,0
9,0.10447,0.38208,0.5,0
20,-0.29582,0.39149,0.5,0
81,-0.29582,0.39149,0.5,0
35,-0.29582,0.39149,0.5,0
86,1.0,0.40012,0.5,0
}\slicestandard

\begin{tikzpicture}[scale=0.75]
  \begin{axis}[
      cycle list/Dark2,
      thick,
      xlabel=$x$,
      grid=both,
      minor grid style={gray!25},
      major grid style={gray!25},
      legend style={at={(0.0, 0.5)},anchor=south west},
    ]
    \addplot+[color=blue] table[x=Points:0, y=u, col sep=comma]{\slicenodal};
    \addlegendentry{$\left(u_h^n\right)^+$};
    \addplot+[color=red, dashed] table[x=Points:0, y=u, col sep=comma]{\slicestandard};
    \addlegendentry{$u_h^n$};

  \end{axis}
\end{tikzpicture}

}
  \caption{Snapshots at final time of the nodally non-negative solutions from the example in Section \ref{ex:drift-diffusion-structure-preservation} and comparisons to the corresponding solutions to (\ref{eq:dg-disc}).}
  \label{fig:drift-diffusion-snapshots}
\end{figure}

The energy change
\begin{equation}
  \label{eq:modified-energy}
  E(w_h^n) := \left(\half\norm{w_h^n}{\Ltwo{\domain}}^2 - \half\norm{w_h^{n-1}}{\Ltwo{\domain}}^2 - \timestep\int_{\bound}\left(w_h^n\nabla w_h^n\cdot\normal + \half \left(w_h^n\right)^2\left(\nabla\elpot\cdot\normal\right)\right)\d s\right)^{1/2},
\end{equation}
which arises from integrating by parts during the energy argument and accounts for the non-homogeneous boundary conditions, is plotted over time in Figure \ref{fig:energy-decay}. We see the expected behaviour of monotonic energy decay when the condition $\Delta\elpot \leq 0$ is satisfied, in agreement with Lemma \ref{lem:dg-disc-energy} and Lemma \ref{lem:drift-diffusion-full-disc-energy}.
\begin{figure}[h]
  \centering
  \input{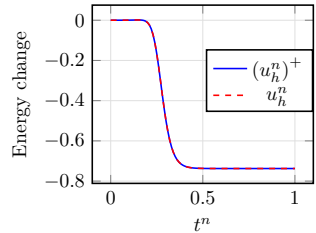}
  \caption{Evolution of the energy change (\ref{eq:modified-energy}) of the solution from Section \ref{ex:drift-diffusion-structure-preservation} with $\elpot = \elpot_1$, where $\Delta\elpot_1 \leq 0$. The results agree with Lemma \ref{lem:dg-disc-energy} and Lemma \ref{lem:drift-diffusion-full-disc-energy}.}
  \label{fig:energy-decay}
\end{figure}

\subsection{An application to the coupled Poisson-Nernst-Planck system}
\label{ex:PNP-structure-preservation}

In the final example we demonstrate the efficacy of our method as an approach to discretising the coupled Poisson-Nernst-Planck system (\ref{eq:PNP-strong}) discussed in the introduction. In this case $\elpot$ is a variable at each time step, and so we introduce $\beta_h(\cdot; \cdot, \cdot)$, for all $\Psi\in\Hone{\mathcal{T}}$, $w, v\in\H{2}{\mathcal{T}}$,
\begin{equation}
  \beta_h(\Psi; w,v)
  :=
  \sum_{K\in\mathcal T}
  \qp{\int_K w \nabla \Psi \cdot \nabla v  \d \vec x}
  -
  \int_{\mathcal{E}_{\rm{int}}}
  \qp{
    \left(\nabla \Psi \cdot \normal\right) \avg{w}\ujump{v}
    -
    \frac{\upen}{2}\lvert\nabla\Psi\cdot\normal\rvert\ujump{w}\ujump{v}
  }\d s.
\end{equation}
Notice that $\beta_h(\psi; w, v) = b_h(w,v)$. We then let
\begin{equation}
  \mathfrak A_h(\Psi; w,v)
  :=
  \ip{w}{v}
  +
  \tau \qp{
    a_h(w, v)
    +
    \beta_h(\Psi; w, v)
  },
\end{equation}
the proposed numerical method is as follows. Given $\varepsilon > 0$ and $f, \posinit, \neginit\in\L\infty{\domain}$, so $\qp{\finpos^0}^+ = \Pi_h\!\left(\posinit\right)$ and $\qp{\finneg^0}^+ = \Pi_h\!\left(\neginit\right)$, for $n = 1, 2,\ldots, N$, seek $\left(\qp{\finpos^n}^+, \qp{\finneg^n}^+, \finelpot^n\right)\in\finspace_p^+\times\finspace_p^+\times\finspace_p$, such that
\begin{align}
  \mathfrak A_h\qp{\finelpot^{n-1}; \qp{\finpos^n}^+, \testpos - \qp{\finpos^n}^+} &\geq \ip{\qp{\finpos^{n-1}}^+}{\testpos - \qp{\finpos^n}^+}\quad\forall\testpos\in\finspace_p^+, \label{eq:PNP-pos} \\
  \mathfrak A_h\qp{\finelpot^{n-1}; \qp{\finneg^n}^+, \testneg - \qp{\finneg^n}^+} &\geq \ip{\qp{\finneg^{n-1}}^+}{\testneg - \qp{\finneg^n}^+}\quad\forall\testneg\in\finspace_p^+, \label{eq:PNP-neg} \\
  \perm\finforma{\finelpot^n}{\testelpot} &= \ip{\qp{\finpos^n}^+ - \qp{\finneg^n}^+ + f}{\testelpot} \Foreach\testelpot\in\finspace_p. \label{eq:PNP-elpot}
\end{align}

By taking the value of $\finelpot$ at the previous time step, equations (\ref{eq:PNP-pos}) and (\ref{eq:PNP-neg}) are instances of the nodally non-negative discretisation of the drift-diffusion equation given by (\ref{eq:drift-diffusion-full-disc}). As a consequence the system can be decoupled by first solving (\ref{eq:PNP-pos}) and (\ref{eq:PNP-neg}), and then (\ref{eq:PNP-elpot}). At the initial time step only (\ref{eq:PNP-elpot}) need be solved.

We test the method (\ref{eq:PNP-pos})--(\ref{eq:PNP-elpot}) over the time interval $\left(0, 0.3\right]$ on a uniform criss-cross triangulation of the unit square with piecewise linear elements, resulting in $360,000$ combined degrees of freedom for all variables. Fixing $\timestep = 5\times10^{-3}$ and $\pen = 100$, we take $\perm = 3\times10^{-4}$, and
\begin{align}
  &f\!\left(x\right) :=
    \begin{cases}
      -1, &x < 1/2, \\
      1, &x \geq 1/2,
    \end{cases}
  &&\posinit = \neginit := \frac{5}{2}\exp{-\left(8\left(x - \half\right)\right)^2 - \left(8\left(y - \half\right)\right)^2}.
\end{align}

Snapshots of the solution are shown in Figure \ref{fig:PNP-snapshots}. A supplementary video is included in the online submisssion of this work for more detail. After initially diffusing from the Gaussian initial condition, $\qp{\finpos}^+$ and $\qp{\finneg}^+$ develop localised regions of high concentration which oscillate between the left and right halves of the domain, increasing in magnitude and decreasing in area until reaching a peak at $t^n = 0.05$. The solutions then decay to zero. In Figure \ref{fig:PNP-minimum} the minimum nodal values of the concentrations are compared against a version of (\ref{eq:PNP-pos})--(\ref{eq:PNP-elpot}) which does not preserve nodal non-negativity -- that is, where (\ref{eq:PNP-pos}) and (\ref{eq:PNP-neg}) are instances of (\ref{eq:dg-disc}) rather than (\ref{eq:drift-diffusion-full-disc}). The method which does not preserve nodal non-negativity experiences significant oscillations and negative values around $t^n = 0.06$. By design the nodally non-negative solution does not become negative.

\begin{figure}[h]
  \centering
  \subcaptionbox{$t^n = 0.005$}
  [0.19\textwidth]{
    \includegraphics[width=\linewidth]{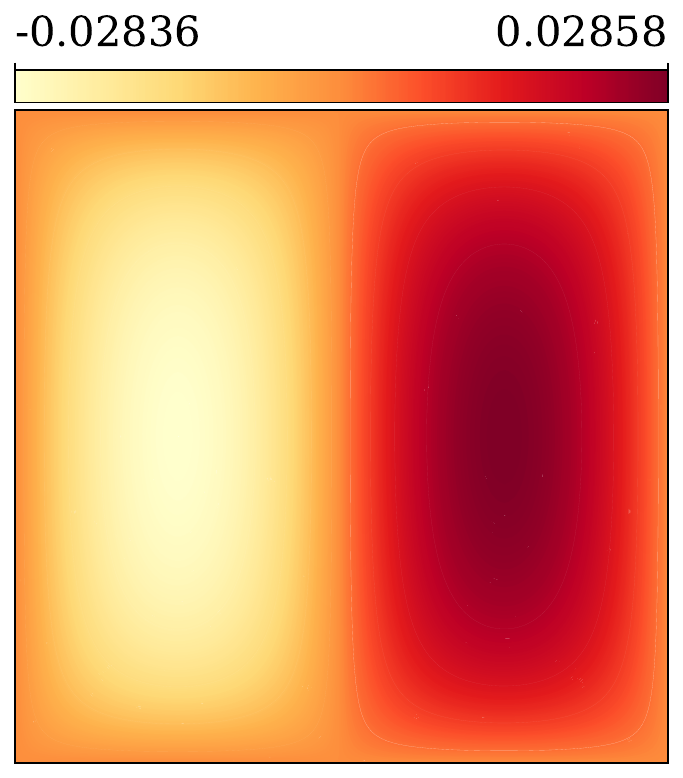}
    \includegraphics[width=\linewidth]{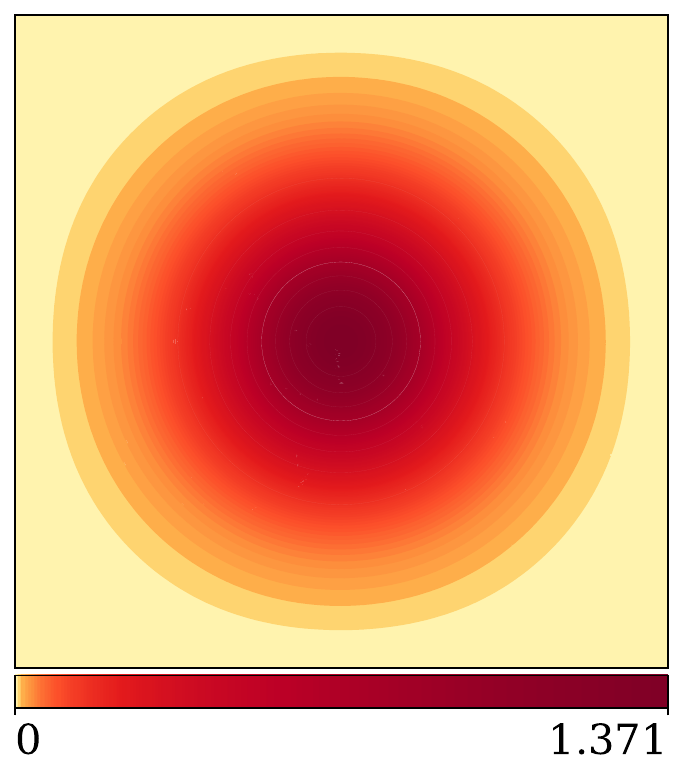}
  }
  \hfill
  \subcaptionbox{$t^n = 0.04$}
  [0.19\textwidth]{
    \includegraphics[width=\linewidth]{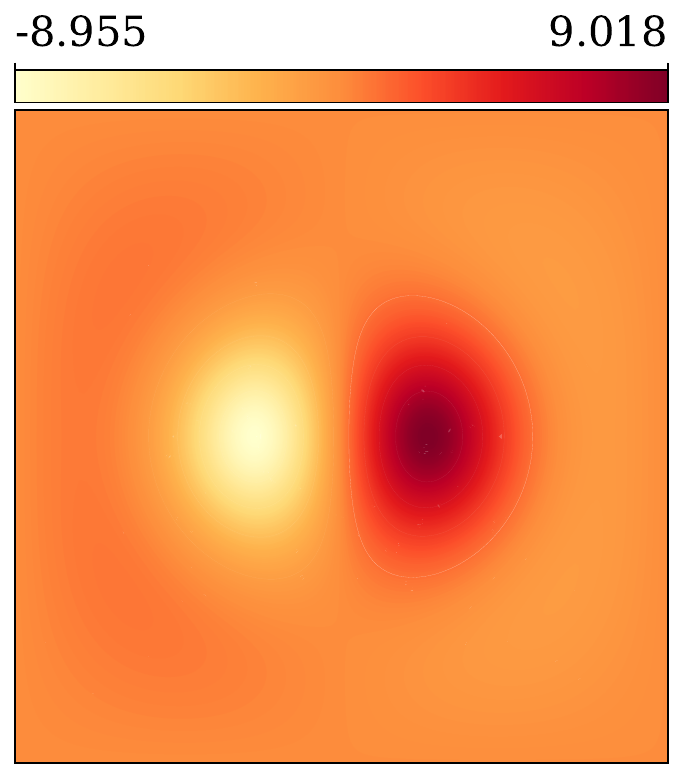}
    \includegraphics[width=\linewidth]{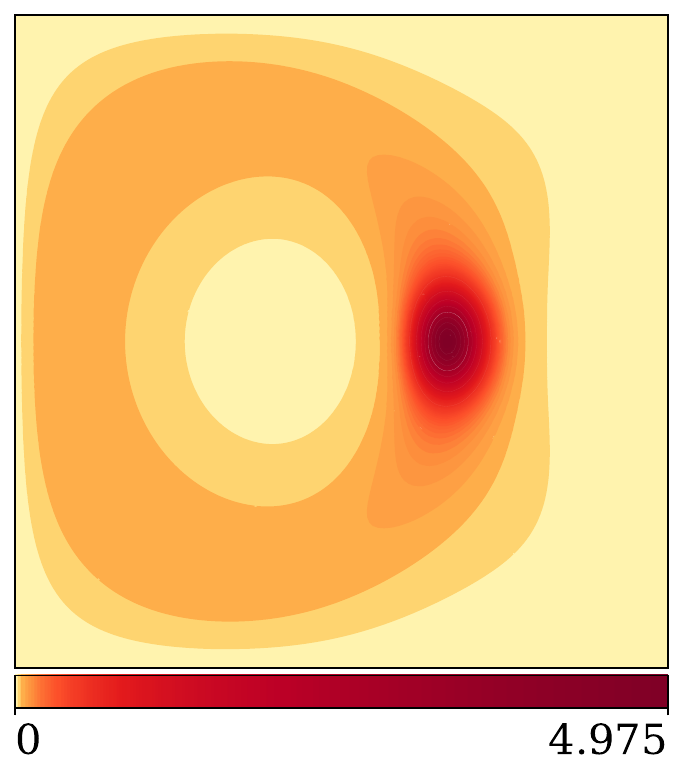}
  }
  \hfill
  \subcaptionbox{$t^n = 0.05$}
  [0.19\textwidth]{
    \includegraphics[width=\linewidth]{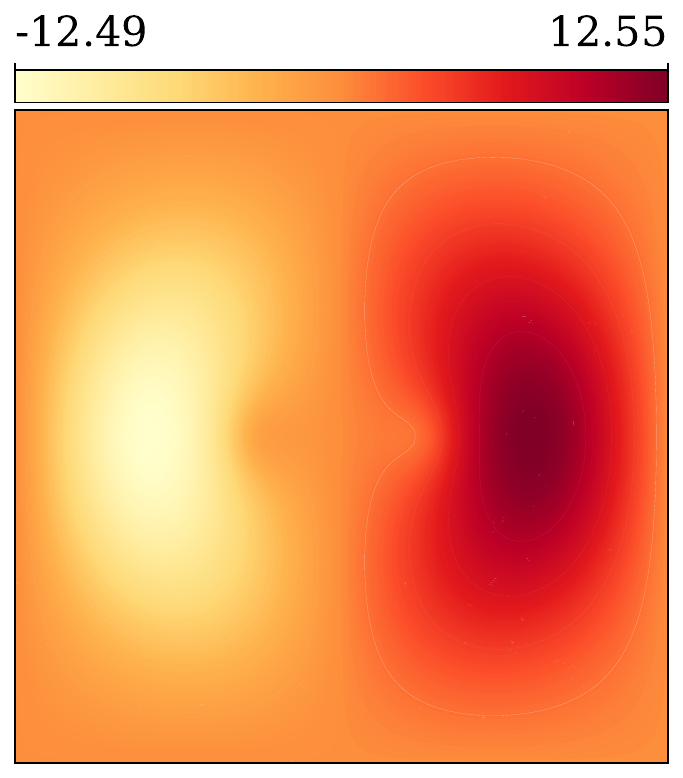}
    \includegraphics[width=\linewidth]{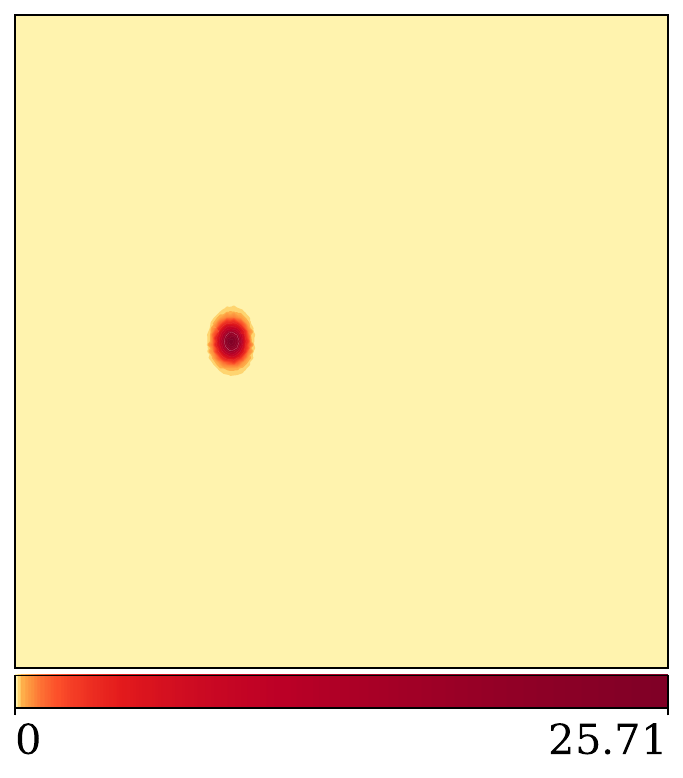}
  }
  \hfill
  \subcaptionbox{$t^n = 0.06$}
  [0.19\textwidth]{
    \includegraphics[width=\linewidth]{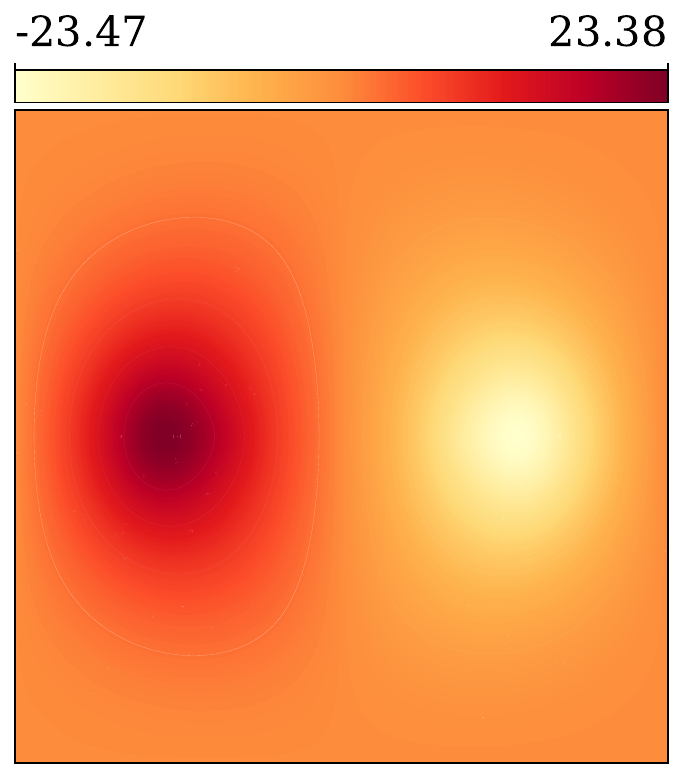}
    \includegraphics[width=\linewidth]{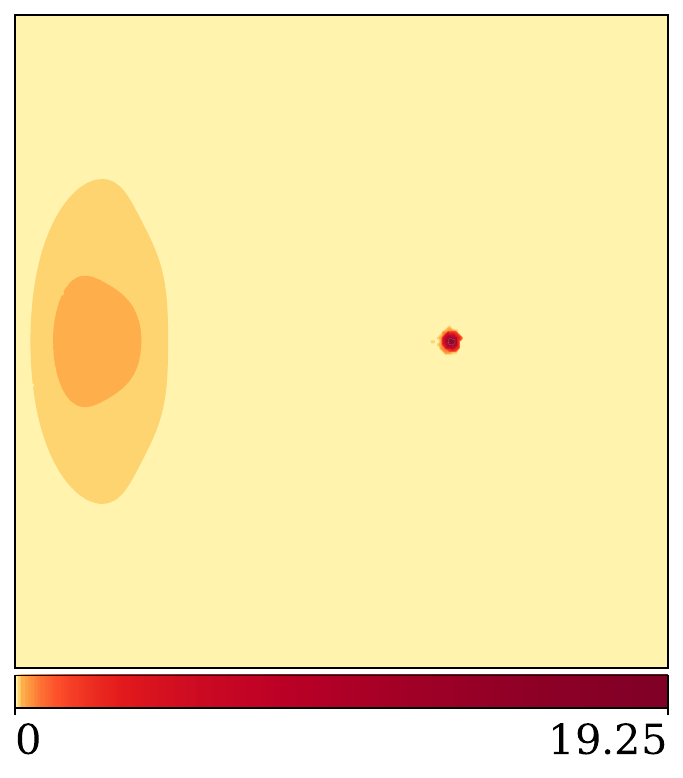}
  }
  \hfill
  \subcaptionbox{$t^n = 0.07$}
  [0.19\textwidth]{
    \includegraphics[width=\linewidth]{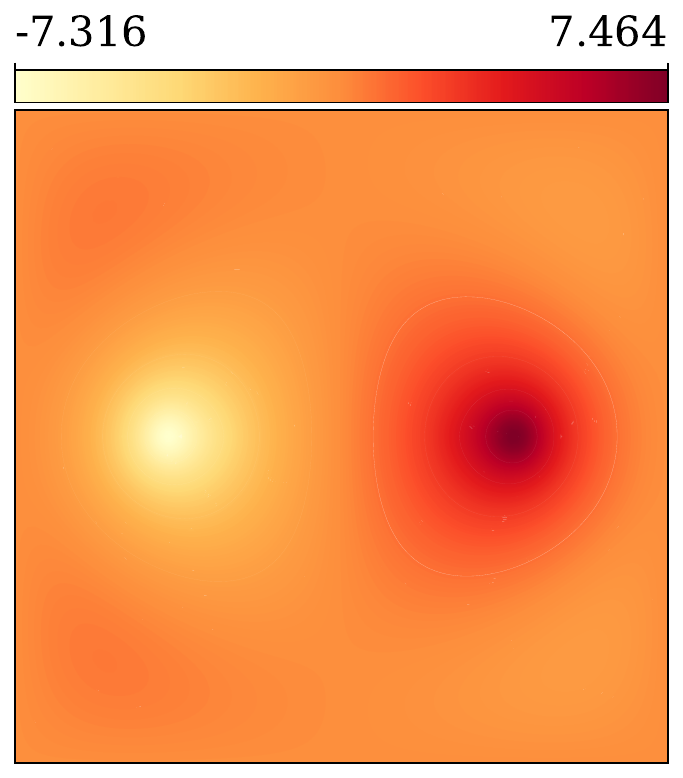}
    \includegraphics[width=\linewidth]{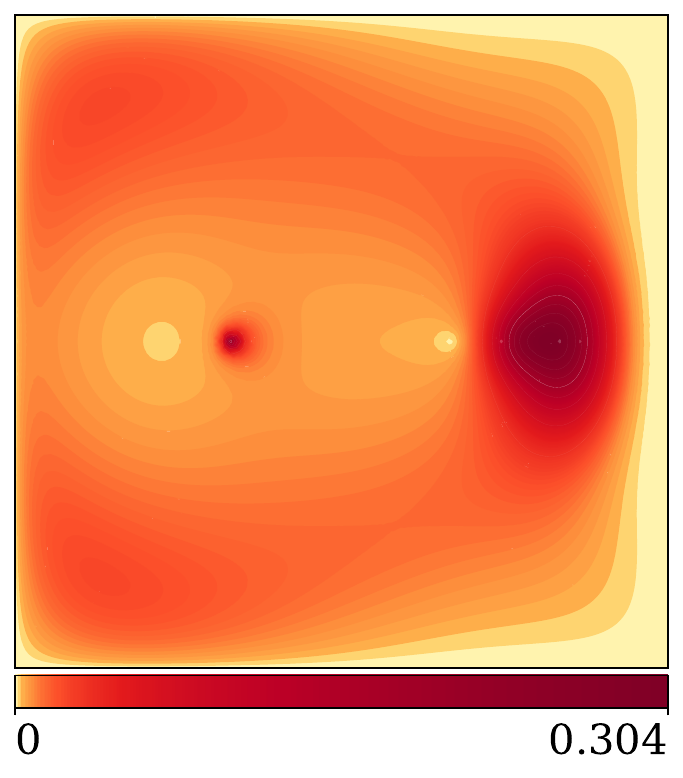}
  }
  \caption{Snapshots of (top) $\finelpot$ and (bottom) $\qp{\finpos}^+$ ($\qp{\finneg}^+$ is similar but reflected in the $y$-direction) from the numerical experiment in Section \ref{ex:PNP-structure-preservation}. We highlight that each subfigure has an individual colour bar and the plots for $\qp{\finpos}^+$ are plotted on a logarithmic scale.}
  \label{fig:PNP-snapshots}
\end{figure}

\begin{figure}[h]
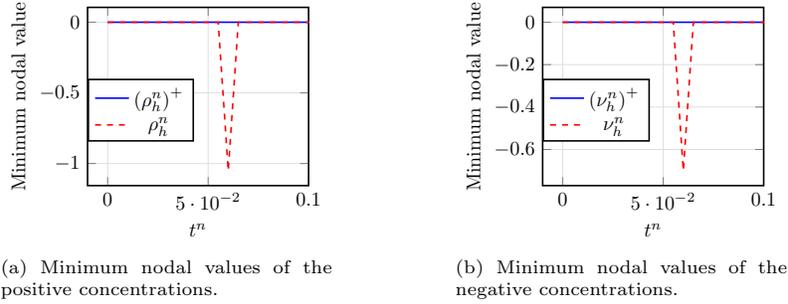

  \centering
  \subcaptionbox{Minimum nodal values of the positive concentrations.}
  {\input{PNP_minimum_pos}}
  \hspace{4em}
  \subcaptionbox{Minimum nodal values of the negative concentrations.}
  {\input{PNP_minimum_neg}}
  \caption{Evolution of the minimum nodal values of the solution to the example from Section \ref{ex:PNP-structure-preservation}. The values for the unconstrained solutions are also plotted for comparison.}
  \label{fig:PNP-minimum}
\end{figure}

\section{Concluding remarks}
\label{sec:conclusions}

The dG methods introduced and analysed in this work are shown to converge optimally in the energy norm, and the nodally bound-preserving method respects a discrete analogy of the maximum principle enjoyed by the drift-diffusion equation. Both discretisations also satisfy an energy dissipation law under the same assumptions on the data as the PDE. Although we focus on a backward Euler time discretisation, we also demonstrate structure preservation for a Crank--Nicolson scheme and discuss other higher-order-in-time methods. The presentation of the nodally bound-preserving method in a variational inequality framework simplifies the treatment of the analysis, and the implementation is achieved via a straightforward iterative projection method at each time step. We summarise with extensive numerical experiments, which demonstrate the just-described properties and the effectiveness of the methods. Finally, with a view towards future work, we consider an extension to the coupled Poisson-Nernst-Planck system, showing the robustness of the bound preservation in this case.

\section*{Funding sources}

AT is supported by a scholarship from the EPSRC Centre for Doctoral
Training in Advanced Automotive Propulsion Systems (AAPS), under the
project EP/S023364/1. TP is grateful for partial support from the
EPSRC grants EP/X030067/1, EP/W026899/1. Both TP and GRB are supported
by the Leverhulme Trust Research Project Grant RPG-2021-238.


\bibliographystyle{alpha}
\bibliography{standard_pde_fem, DG, software_packages, PNP, BSFEM, batteries_chemistry, NSPNP}

\end{document}